\documentclass{amsart}
\usepackage{amsfonts, amsmath, amsthm, amssymb}
\usepackage[pdftex]{graphicx}
\usepackage{commath}
\usepackage{wrapfig,lipsum}
\usepackage{float}
\usepackage{caption}
\usepackage{subcaption}
\usepackage[english]{babel}
\usepackage{hyperref}
\usepackage[utf8]{inputenc}
\usepackage{tikz}
\usepackage{filecontents}
\usepackage[bottom]{footmisc}
\AtBeginDocument{}

\usepackage{enumerate}
\newtheorem*{theorem*}{Theorem}
\newtheorem{theorem}{Theorem}[section]
\newtheorem{lemma}[theorem]{Lemma}
\newtheorem{proposition}[theorem]{Proposition}
\newtheorem*{proposition*}{Proposition}

\theoremstyle{definition}
\newtheorem{definition}[theorem]{Definition}

\theoremstyle{remark}
\newtheorem{remark}[theorem]{Remark}
\DeclareMathAlphabet{\mathpzc}{OT1}{pzc}{m}{it}

\def\Id{\operatorname{Id}}

\def\Im{\operatorname{Im}}
\def\st{\operatorname{st}}

\def\PU{\operatorname{PU}}
\def\U{\operatorname{U}}

\def\Re{\operatorname{Re}}
\def\Aut{\operatorname{Aut}}

\def\z{\operatorname{z}}

\def\dist{\operatorname{dist}}
\def\ker{\operatorname{Ker}}

\def\cosh{\operatorname{cosh}}

\def\dist{\operatorname{dist}}

\usepackage{filecontents}
\usepackage[bottom]{footmisc}
\usepackage{amsmath}
\usepackage{amssymb}
\usepackage{amsthm}
\usepackage{xcolor}
\usepackage{mathtools}
\usepackage{breqn}
\def\c{\mathbb C}
\def\p{\mathbb P}
\def\r{\mathbb R}
\def\z{\mathbb Z}

\def\chp{H_{\c}^2}

\def\bchp{{\partial\chp}}

\def\calC{\mathcal{C}}
\def\calE{\mathcal{E}}
\def\calN{\mathcal{N}}
\def\calP{\mathcal{P}}
\def\calS{\mathcal{S}}
\def\calT{\mathcal{T}}

\def\al{{\alpha}}

\def\Ga{{\Gamma}}
\def\ga{{\gamma}}

\def\La{{\Lambda}}
\def\Si{{\Sigma}}
\def\si{{\sigma}}

\def\Re{\operatorname{Re}}
\def\Im{\operatorname{Im}}

\def\EE{E}

\def\st{\,\,\big|\,\,}
\def\<{\langle}
\def\>{\rangle}

\let\ge=\geqslant
\let\le=\leqslant

\font\msbm=msbm10
\def\semiprod{\hbox{\msbm\char111}}

\def\defit{\it}

\title[Complex Hyperbolic Triangle Groups of Type ${[}m,m,0;n_1,n_2,2{]}$]{Complex Hyperbolic Triangle Groups\\ of Type $[m,m,0;n_1,n_2,2]$}
\author[Sam Povall]{Sam Povall}
\address{School of Engineering\\ University of Liverpool\\ Liverpool L69~7ZL, United Kingdom}
\email{S.Povall@liverpool.ac.uk}
\author[Anna Pratoussevitch]{Anna Pratoussevitch}
\address{Department of Mathematical Sciences\\ University of Liverpool\\ Liverpool L69~7ZL, United Kingdom}
\email{annap@liverpool.ac.uk}
\begin{date}  {\today} \end{date}
\DeclareMathAlphabet{\mathpzc}{OT1}{pzc}{m}{it}
\begin{document}
\maketitle
\begin{abstract}
In this paper we study discreteness of complex hyperbolic triangle groups of type \([m,m,0; n_1, n_2, 2]\), i.e. groups of isometries of the complex hyperbolic plane generated by three complex reflections of orders \(n_1, n_2, 2\) in complex geodesics with pairwise distances \(m,m,0\). For fixed \(m\), the parameter space of such groups is of real dimension one. We determine the possible orders for \(n_1\) and \(n_2\) and also intervals in the parameter space that correspond to discrete and non-discrete triangle groups.
\end{abstract}
\section{Introduction}
\noindent
Complex hyperbolic triangle groups are groups of isometries of the complex hyperbolic plane
generated by three complex reflections in complex geodesics. We will focus on the case of ultra-parallel groups, that is, the case where the complex geodesics are pairwise disjoint. 
\\
\linebreak
Unlike real reflections, complex reflections can be of arbitrary order. Work on higher order reflections has been discussed by Parker and Paupert \cite{parkpau} and Pratoussevitch \cite{anna}. If an ultra-parallel complex hyperbolic triangle group is generated by reflections of orders $n_1,n_2,n_3$ in complex geodesics $C_1,C_2,C_3$ with the distance between~$C_{k-1}$ and~$C_{k+1}$ equal to~$m_k$ for~$k=1,2,3$, then we say that the group is of type $[m_1,m_2,m_3;n_1,n_2,n_3]$. 
\\
\linebreak
In this paper, we will study discreteness of ultra-parallel complex hyperbolic triangle groups of type $[m,m,0;n_1,n_2,2]$, i.e.\ one reflection of order~$n_1$, one of order \(n_2\), and one of order~$2$. The fixed point sets of order~$n_1$ and \(n_2\) reflections intersect on the boundary of the complex hyperbolic plane ($m_3=0$) and the other two distances between fixed point sets coincide ($m_1=m_2$). Ultra-parallel triangle groups of types $[m,m,0;2,2,2]$ and $[m,m,2m;2,2,2]$ have been considered in \cite{WyssGall}, groups of type $[m_1,m_2,0;2,2,2]$ have been considered in \cite{andyanna} and \cite{andy}, and groups of type $[m,m,0;3,3,2]$ have been considered in~\cite{pov1} and~\cite{povprat}. 
\\
\linebreak
To determine the possible orders~$n_1$ and~$n_2$ of the complex reflections for the groups of type $[m,m,0;n_1,n_2,2]$ to be discrete,
we use the work of Hersonsky and Paulin~\cite{hersonpaul}.
We combine several of their results which give rise to the following theorem:

\begin{theorem}
\label{reforders}
An ultra-parallel complex hyperbolic triangle group of type 
\\
$[m_1,m_2,0;n_1,n_2,n_3]$ can only be discrete if the unordered pair of orders of the complex reflections~$\iota_1$ and~$\iota_2$ is one of
\[
  \{2,2\},\{2,3\},\{2,4\},\{2,6\},\{3,3\},\{3,6\}~\mbox{or}~\{4,4\}.
\]
\end{theorem}

\noindent
The deformation space of groups of type $[m,m,0;n_1,n_2,2]$ for a given $m$ is of real dimension one. A group is determined up to an isometry by the angular invariant $\alpha\in[0,2\pi]$, see section~\ref{sec-background}. The main aim is to determine an interval in this one-dimensional deformation space such that for all values of the angular invariant in this interval the corresponding triangle group is discrete. The main result of the paper is the following proposition:

\begin{proposition}
\label{prop-discr}
Let $[m_1,m_2,m_3;n_1,n_2,n_3]$ be the type of a complex hyperbolic triangle group~$\Ga$.
Assume that $m_1=m_2=m$, $m_3=0$, $n_1\le n_2$ and~$n_3=2$.
Then $\Ga$ is discrete if 
\[\cosh(m/2)\ge\frac{1}{\sin(\pi/n_2)}\]
and the angular invariant~$\al$ of~$\Ga$ satisfies $\text{\rm ctan}(\al/2)\le C$, where
\begin{center}
\def\arraystretch{1.5}
\begin{tabular}{| c || c | c | c | c | c | c | c | c |}
\hline
$(n_1,n_2)$ & $(3,3)$ & $(2,3)$ & $(2,6)$ & $(3,6)$ & $(2,4)$ & $(4,4)$ \\
\hline
\hline
$C$ {\rm(exact value)} & $1/\sqrt{3}$ & $2\sqrt{2}-\sqrt{3}$ & $(4-\sqrt{5})/\sqrt{3}$ & $1/\sqrt{3}$ & $1$ & $2-\sqrt{3}$ \\
\hline
$C$ {\rm(approx value)} & $0{.}577$ & $1{.}096$ & $1{.}018$ & $0{.}577$ & $1$ & $0{.}268$ \\
\hline
\end{tabular}
\end{center}
\end{proposition}

%\[\cosh(m_k/2)\ge\frac{1}{\sin(\pi/n_k)}\quad\text{for}~k=1,2,3\]

%m_k\ge\log_e\left(\frac{1+\cos(\pi/n_k)}{1-\cos(\pi/n_k)}\right) \quad\text{for}~k=1,2,3

%Note that
%\[m_k\ge\log_e\left(\frac{1+\cos(\pi/n_k)}{1-\cos(\pi/n_k)}\right)\]
%implies that 
%\[r_k=\cosh(m_k/2)\ge\frac{1}{\sin(\pi/n_k)}.\]

%The lower bound on~$m$ can be expressed as
%\begin{enumerate}[$\bullet$]
%\item
%$\log_e(3)$ for~$n_2=3$,
%\item
%$\log_e\left(3+2\sqrt{2}\right)$ for~$n_2=4$,
%\item
%$\log_e\left(7+4\sqrt{3}\right)$ for~$n_2=6$.
%\end{enumerate}

\noindent
To prove this proposition, we use a version of Klein's combination theorem, adapted to the configurations in question. Two of the generating reflections share a fixed point on the boundary of the complex hyperbolic plane. We show that the ultra-parallel triangle group satisfies a compression property by carefully studying the structure of the stabilizer of this fixed point and of its subgroup of Heisenberg translations.
\\
\linebreak
On the other hand, we obtain the following non-discreteness results by applying theorems from \cite{parkerford} and \cite{parkershimizu}:

\begin{proposition}
\label{prop-nondiscr}
Let $[m_1,m_2,m_3;n_1,n_2,n_3]$ be the type of a complex hyperbolic triangle group~$\Ga$.
Assume that $m_3=0$, and~$n_3=2$.
Let $\al$ be the angular invariant of~$\Ga$.
Then $\Ga$ is non-discrete if
\[4\cosh\left(\frac{m_1}{2}\right)\cosh\left(\frac{m_2}{2}\right)\sin^2\left(\frac{\alpha}{2}\right)+\left(\cosh\left(\frac{m_1}{2}\right)-\cosh\left(\frac{m_2}{2}\right)\right)^2<\frac{2}{\nu_1}\]
and
\[4\cosh\left(\frac{m_1}{2}\right)\cosh\left(\frac{m_2}{2}\right)\sin^2\left(\frac{\alpha}{2}\right)+\left(\cosh\left(\frac{m_1}{2}\right)-\cosh\left(\frac{m_2}{2}\right)\right)^2\ne\frac{2\cos(\pi/q)}{\nu_1},\]
where $\nu_1$ and $q$ are given by
\begin{center}
\def\arraystretch{1.5}
\begin{tabular}{| c || c | c | c | c | c | c | c | c}
\hline
$(n_1,n_2)$ & $(3,3)$ & $(2,3)$ & $(2,6)$ & $(3,6)$ & $(2,4)$ & $(4,4)$ \\
\hline
\hline
$\nu_1$ & $6\sqrt{3}$ & $24\sqrt{3}$ & $4\sqrt{3}$ & $2\sqrt{3}$ & $16$ & $4$ \\
\hline
$q\in\z$ & $q\geq 6$ & $q\geq 6$ & $q\geq 4$ & $q\geq 3$ & $q\geq 3$ & $q\geq 4$ \\
\hline
\end{tabular}
\end{center}
\end{proposition}

\bigskip\noindent
Combining the results of Propositions~\ref{prop-discr} and~\ref{prop-nondiscr}, we see that there is a gap between the intervals of discreteness and non-discreteness. This is illustrated in Figure~\ref{fig-gap} for the case \([m,m,0; 2,3,2]\). The figure shows the $(m,\al)$-space. The light grey box corresponds to discrete groups (Proposition~\ref{prop-discr}). The solid and dotted black areas correspond to non-discrete groups (Proposition~\ref{prop-nondiscr}).
% FIGURE
\begin{figure}[h]
\begin{center}
\begin{tikzpicture}
%Draw the outside box;
\path[draw] (0.8,1)--(11.3,1);
\path[draw] (0.8,7.014)--(11,7.014);
\path[draw] (1,1)--(1,7.5);
\path[draw] (1,4)--(0.8,4);
%The arrow at 'm';
\path[draw] (11.1,1.2)--(11.3,1);
\path[draw] (11.1,0.8)--(11.3,1);
%The arrow at 'alpha'
\path[draw] (0.8,7.3)--(1,7.5);
\path[draw] (1.2,7.3)--(1,7.5);
%Adding labels to graph;
\node at (1,7.8) {$\alpha$};
\node at (11.6,1) {$m$};
\node at (0.6,4) {$\pi$};
\node at (0.6,7.014) {$2\pi$};
\node at (0.5,1) {$0$};
%Dotted lines to angles and labels to 4pi/3;
\path[draw, dashed] (1,5)--(3,5);
\path[draw, dashed] (1,3)--(3,3);
\node at (0.6,5) {$\frac{4\pi}{3}$};
\node at (0.6,3) {$\frac{2\pi}{3}$};
%Dotted line to m and label to log(3);
\path[draw, dashed] (3,0.8)--(3,3);
\node at (3,0.5) {$\mbox{log}_e(3)$};
%Discreteness box with colour;
\path[draw] (3,5)--(11,5);
\path[draw] (3,3)--(11,3);
\path[draw] (3,3)--(3,5);
\draw[fill= lightgray]  (3,3) -- (11,3) -- (11,5) -- (3,5) -- cycle;
%Non-discreteness region with colour;
\draw[fill=darkgray,darkgray] (1,6.45) .. controls (3,6.4) .. (11,6.9);
\path[draw, loosely dotted] (1,6.2) .. controls (3,6.15) .. (11,6.65);
\path[draw, dotted] (1,6.25) .. controls (3,6.2) .. (11,6.7);
\path[draw, loosely dotted] (1,6.3) .. controls (3,6.25) .. (11,6.75);
\path[draw, dotted] (1,6.35) .. controls (3,6.3) .. (11,6.8);
\path[draw, loosely dotted] (1,6.4) .. controls (3,6.35) .. (11,6.85);
\draw (1,6.449) .. controls (3,6.395) .. (11,6.895);
\draw[fill=darkgray,darkgray]  (1.015,6.45) -- (11,6.9) -- (11,7) -- (1.015,7) -- cycle;
\draw[fill=darkgray,darkgray] (1,1.55) .. controls (3,1.6) .. (11,1.1);
\path[draw, loosely dotted] (1,1.8) .. controls (3,1.85) .. (11,1.35);
\path[draw, dotted] (1,1.75) .. controls (3,1.8) .. (11,1.3);
\path[draw, loosely dotted] (1,1.7) .. controls (3,1.75) .. (11,1.25);
\path[draw, dotted] (1,1.65) .. controls (3,1.7) .. (11,1.2);
\path[draw, loosely dotted] (1,1.6) .. controls (3,1.65) .. (11,1.15);
\draw (1,1.551) .. controls (3,1.601) .. (11,1.101);
\draw[fill=darkgray,darkgray]  (1.015,1.55) -- (11,1.1) -- (11,1.015) -- (1.015,1.015) -- cycle;
\end{tikzpicture}
\end{center}
\caption{Discreteness and non-discreteness results in the $(m,\al)$-space.}
\label{fig-gap}
\end{figure}
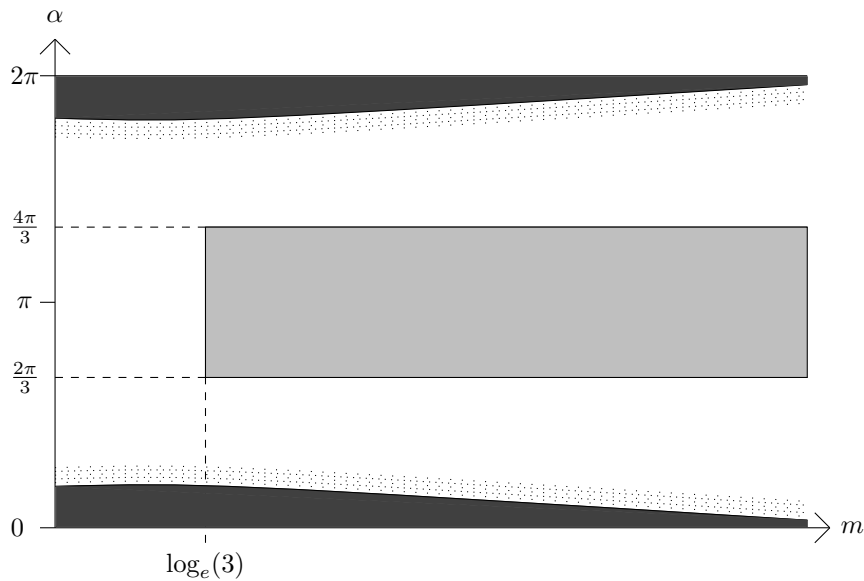
\noindent
Ultra-parallel complex hyperbolic triangle groups of type $[m,m,0;n_1,n_2,2]$ with orders $(n_1,n_2)$ other than $(2,2)$ were first considered in \cite{pov1} and~\cite{povprat}.
In the case $(n_1,n_2)=(3,3)$,
the result in Proposition~\ref{prop-discr} repeats the result of~\cite{povprat}, 
while the result in Proposition~\ref{prop-nondiscr} improves the result of~\cite{povprat}.
\\
\linebreak
The paper is organised as follows:
In section~\ref{sec-background} we discuss the background information on complex hyperbolic and Heisenberg geometry. We then introduce the standard parametrisation for ultra-parallel $[m_1,m_2,0;n_1,n_2,n_3]$-triangle groups in section~\ref{sec-param}.
In section~\ref{ordreflc}, we conclude the possible orders of the complex reflections for the groups of type $[m_1,m_2,0;n_1,n_2,n_3]$ to be discrete.
Following this,
in section~\ref{sec-compression} we use the compression property to derive a discreteness condition for $[m_1,m_2,0;n_1,n_2,n_3]$-groups.
In section~\ref{param} we specialise the standard parametrisation to the case of ultra-parallel $[m,m,0;n_1,n_2,2]$-triangle groups.
In section~\ref{heitran} we study the structure of the stabilizer of the intersection points, that is, the intersection of the fixed point sets of orders \(n_1\) and \(n_2\) with the boundary of the complex hyperbolic plane.
In section~\ref{disres} we use the discreteness conditions from section~\ref{sec-compression} to give a proof of Proposition~\ref{prop-discr}.
Finally in section~\ref{nondisres} we contrast the discreteness results with non-discreteness results and prove Proposition \ref{prop-nondiscr}.

\begin{remark}
We use the following notation: For group elements~$A$ and~$B$, their commutator is $[A,B]=A^{-1}B^{-1}AB$.
\end{remark}

\section{Background}
\label{sec-background}
\noindent
This section will give a brief introduction to complex hyperbolic geometry, for further details see \cite{goldman,parker2003}.

\subsection{Complex hyperbolic plane:}
Let $\c^{2,1}$ be the $3$-dimensional complex vector space
equipped with a Hermitian form $\<\cdot,\cdot\>$ of signature $(2,1)$,
e.g.
\[\<z,w\>=z_1\bar{w}_1+z_2\bar{w}_2-z_3\bar{w}_3.\]
If $z\in\c^{2, 1}$ then we know that $\<z, z\>$ is real.
Thus we can define subsets $V_-$, $V_0$ and $V_+$ of $\c^{2,1}$ as follows
\begin{align*}
  V_-&=\{z\in\c^{2,1}\st\<z,z\><0\},\\
  V_0&=\{z\in\c^{2,1}\backslash\{0\}\st\<z,z\>=0\},\\
  V_+&=\{z\in\c^{2,1}\st\<z,z\>>0\}.
\end{align*}
We say that $z\in\c^{2,1}$ is {\defit negative\/}, {\defit null\/} or {\defit positive\/} if $z$ is in $V_-$, $V_0$ or $V_+$ respectively. Define a projection map $\p$ on the points of $\c^{2,1}$ with $z_3\ne0$ as 
\[\p : z=\begin{bmatrix} z_1\\ z_2\\ z_3\end{bmatrix}\mapsto\begin{pmatrix} z_1/z_3\\ z_2/z_3\end{pmatrix}\in\p(\c^{2,1}).\]
That is, provided $z_3\ne0$, 
\[z=(z_1, z_2, z_3)\mapsto[z]=[z_1:z_2:z_3]=\left[\frac{z_1}{z_3}:\frac{z_2}{z_3}:1\right].\]
The {\defit projective model\/} of the complex hyperbolic plane is defined to be the collection of negative lines in $\c^{2,1}$
and its boundary is defined to be the collection of null lines.
That is
\[\chp=\p(V_-)\quad\text{and}\quad\bchp=\p(V_0).\]
The metric on $\chp$, called the {\defit Bergman metric\/}, is given by the distance function~$\rho$ defined by the formula
\[
  \cosh^2\left(\frac{\rho([z],[w])}{2}\right)
  =\frac{\langle{z, w\rangle}\langle{w, z\rangle}}{\langle{z, z\rangle}\langle{w, w\rangle}},
\]
where $[z]$ and $[w]$ are the images of~$z$ and $w$ in $\c^{2,1}$ under the projectivisation map~$\p$. 
The group of holomorphic isometries of~$\chp$ with respect to the Bergman metric
can be identified with the projective unitary group $\PU(2,1)$.

\subsection{Complex geodesics:}
A {\defit complex geodesic\/} is a projectivisation of a 2-dimensional complex subspace of $\c^{2,1}$.
Any complex geodesic is isometric to \[\{[z:0:1]\st z\in\c\}\] in the projective model.
Any positive vector $c\in V_+$ determines a two-dimensional complex subspace
\[\{z\in\c^{2,1}\st \<c,z\>=0\}.\]
Projecting this subspace we obtain a complex geodesic
\[\p\left(\{z\in\c^{2,1}\st \<c,z\>=0\}\right).\]
Conversely, any complex geodesic is represented by a positive vector $c\in V_+$, 
called a {\defit polar vector\/} of the complex geodesic.
A polar vector is unique up to multiplication by a complex scalar.
We say that the polar vector~$c$ is {\defit normalised\/} if $\<c,c\>=1$.
\\
\linebreak
Let $C_1$ and $C_2$ be complex geodesics with normalised polar vectors~$c_1$ and~$c_2$ respectively.
We call $C_1$ and $C_2$ {\defit ultra-parallel\/} if they have no points of intersection in $\chp\cup\bchp$,
in which case
\[|\<c_1,c_2\>|=\cosh\left(\frac{1}{2}\dist(C_1, C_2)\right)>1,\]
where $\dist(C_1, C_2)$ is the distance between $C_1$ and~$C_2$.
We call $C_1$ and $C_2$ {\defit ideal\/} if they have a point of intersection in $\bchp$,
in which case $|\<c_1,c_2\>|=1$ and $\dist(C_1, C_2)=0$.

\subsection{Complex reflections:}
For a given complex geodesic $C$, a {\defit minimal complex hyperbolic reflection of order~$n$} in~$C$
is the isometry $\iota_C$ in $\PU(2,1)$ of order~$n$ with fixed point set~$C$ given by
\[\iota(z) = -z+(1-\mu)\frac{\<z,c\>}{\<c,c\>}c,\]
where $c$ is a polar vector of~$C$ and $\mu=\exp(2\pi i/n)$.

\subsection{Complex hyperbolic triangle groups:}
A {\defit complex hyperbolic triangle\/} is a triple \((C_1, C_2, C_3)\) of complex geodesics in $\chp$.
A triangle \((C_1, C_2, C_3)\) is a {\defit complex hyperbolic ultra-parallel \([m_1, m_2, m_3]\)-triangle\/}
if the complex geodesics are ultra-parallel at distances $m_k=\dist(C_{k-1}, C_{k+1})$ for $k=1,2,3$.
We will allow $m_k=0$ for some or all~$k$.
A {\defit complex hyperbolic ultra-parallel \([m_1,m_2,m_3;n_1,n_2,n_3]\)-triangle group\/}
is a subgroup of \(\PU(2,1)\) generated by complex reflections \(\iota_k\) of order \(n_k\) in the sides \(C_k\)
of a complex hyperbolic ultra-parallel \([m_1, m_2, m_3]\)-triangle \((C_1, C_2, C_3)\).

\subsection{Angular invariant:}
For each fixed triple \(m_1, m_2, m_3\) the space of \([m_1, m_2, m_3]\)-triangles is of real dimension one.
We can describe a parametrisation of the space of complex hyperbolic triangles in $\chp$ by means of an angular invariant $\al$.
We define the {\defit angular invariant\/} $\al$ of the triangle \((C_1, C_2, C_3)\) by
\[\al=\arg\left(\prod_{k=1}^3 \<c_{k-1}, c_{k+1}\>\right),\]
where $c_k$ is the normalised polar vector of the complex geodesic~$C_k$.
We use the following proposition, given in \cite{anna}, which gives criteria for the existence of a triangle group
in terms of the angular invariant.

\begin{proposition}
\label{traingle-existence}
An $[m_1, m_2, m_3]$-triangle in $\chp$ is determined uniquely up to isometry
by the three distances between the complex geodesics and the angular invariant~$\al$.
For any $\al\in[0, 2\pi]$, an $[m_1, m_2, m_3]$-triangle with angular invariant~$\al$ exists if and only if
\[\cos(\al)<\frac{r_1^2+r_2^2+r_3^2-1}{2r_1r_2r_3},\]
where $r_k=\cosh(m_k/2)$.
\end{proposition}

\noindent
For $m_3=0$ we have $r_3=1$ and the right hand side of the inequality in Proposition~\ref{traingle-existence} is 
\[\frac{r_1^2+r_2^2}{2r_1r_2}\ge1,\]
so the condition on~$\al$ is always satisfied,
i.e.\ for any $\al\in[0, 2\pi]$ there exists an $[m_1, m_2, m_3]$-triangle with angular invariant~$\al$.

\subsection{Heisenberg group:}
\label{heisenberg-group}
The boundary of complex hyperbolic space can be identified with the one-point compactification of the {\defit Heisenberg group\/} 
\[\calN=\c\times\r\cup\{\infty\}=\{(\zeta,\nu)\st\zeta\in\c,\nu\in\r\}\cup\{\infty\}.\]
One homeomorphism taking $\bchp$ to $\calN$ is given by the stereographic projection:
\[
  [z_1:z_2:z_3]\mapsto\left(\frac{z_1}{z_2+z_3}, \Im\left(\frac{z_2-z_3}{z_2+z_3}\right)\right)~\text{if}~z_2+z_3\ne0,
  \quad
  [0:z:-z]\mapsto\infty.
\]
The {\defit Heisenberg group\/} is the Heisenberg space~$\calN$ with the group law
\[(\xi_1,\nu_1)*(\xi_2,\nu_2)=(\xi_1+\xi_2,\nu_1+\nu_2+2\Im(\xi_1\bar{\xi_2})).\]
The centre of~$\calN$ consists of elements of the form~$(0,\nu)$ for~$\nu\in\r$.
The Heisenberg group is not abelian but is $2$-step nilpotent.
To see this, observe that 
\[[(\xi_1,\nu_1),(\xi_2,\nu_2)]=(\xi_1,\nu_1)^{-1}*(\xi_2,\nu_2)^{-1}*(\xi_1,\nu_1)*(\xi_2,\nu_2)=(0,4\Im(\xi_1\bar{\xi_2})).\]
Therefore the commutator of any two elements of~$\calN$ lies in the centre.
\\
\linebreak
An alternative description of the Heisenberg group~$\calN$ is as the group of upper triangular matrices
\[\left\{\begin{pmatrix} 1&x&y\\ 0&1&z\\ 0&0&1\end{pmatrix}\st x,y,z\in\r\right\}\]
with the operation of matrix multiplication.
For any integer~$k\ne0$, the subgroup $N_k$ generated by the matrices
\[a=\begin{pmatrix} 1&0&0\\ 0&1&1\\ 0&0&1\end{pmatrix},\quad b=\begin{pmatrix} 1&1&0\\ 0&1&0\\ 0&0&1\end{pmatrix}\quad\text{and}\quad c=\begin{pmatrix} 1&0&\frac{1}{k}\\ 0&1&0\\ 0&0&1\end{pmatrix}\]
is a uniform lattice in~$\calN$ with the presentation
\[N_k=\<a,b,c\st [b,a]=c^k,~[c,a]=[c,b]=1\>.\] 
Moreover, any uniform lattice in~$\calN$ is isomorphic to~$N_k$ for some integer~$k\ne0$, see section~6.1 in~\cite{dekimpe}.

\subsection{Chains:}
A complex geodesic in~$\chp$ is homeomorphic to a disc,
its intersection with the boundary of the complex hyperbolic plane is homeomorphic to a circle.
Circles that arise as the boundaries of complex geodesics are called {\defit chains\/}.
\\
\linebreak
There is a bijection between chains and complex geodesics. We can therefore, without loss of generality, talk about reflections in chains instead of reflections in complex geodesics. 
\\
\linebreak
Chains can be represented in the Heisenberg space, for more details see \cite{goldman}.
Chains passing through~$\infty$ are represented by vertical straight lines defined by $\zeta = \zeta_0$.
Such chains are called {\defit vertical\/}.
The vertical chain $C_{\zeta_0}$ defined by $\zeta=\zeta_0$ has a polar vector
\[c_{\zeta_0}=\begin{bmatrix}1\\ -\bar{\zeta_0}\\ \bar{\zeta_0}\end{bmatrix}.\]
A chain not containing~$\infty$ is called {\defit finite\/}.
A finite chain is represented by an ellipse whose vertical projection $\c\times\r\rightarrow\c$ is a circle in~$\c$.
The finite chain with centre $(\zeta_0,\nu_0)\in\calN$ and radius $r_0 > 0$ has a polar vector
\[\begin{bmatrix}2\zeta_0 \\ 1+r_0^2-\zeta_0\bar{\zeta_0}+i\nu_0 \\ 1-r_0^2+\zeta_0\bar{\zeta_0}-i\nu_0  \end{bmatrix}\]
and consists of all points~$(\zeta,\nu)\in\calN$ satisfying the equations
\[|\zeta-\zeta_0|=r_0,\quad\nu=\nu_0-2\Im(\zeta\bar{\zeta}_0).\]

\subsection{Heisenberg isometries:}
\label{heisenberg-isometries}
We consider the space~$\calN$ equipped with the {\defit Cygan metric\/},
\[
  \rho_0\left((\zeta_1, \nu_1), (\zeta_2, \nu_2)\right)
  = \Big|\abs{\zeta_1-\zeta_2}^2-i(\nu_1-\nu_2)-2i\Im(\zeta_1\bar{\zeta_2})\Big|^{1/2}.
\]

\bigskip\noindent
A {\defit Heisenberg translation\/}~$T_{(\xi,\nu)}$ by $(\xi,\nu)\in\calN$ is given by
\[(\zeta,\omega)\mapsto(\zeta+\xi,\omega+\nu+2\Im(\xi\bar{\zeta}))=(\xi,\nu)*(\zeta,\omega)\]
and corresponds to the following element in $\PU(2,1)$
\[
  \begin{pmatrix}
    1 & \xi & \xi \\
    -\bar{\xi} & 1-\frac{|\xi|^2-i\nu}{2} & -\frac{|\xi|^2-i\nu}{2}\\ 
    \bar{\xi} & \frac{|\xi|^2-i\nu}{2} & 1+\frac{|\xi|^2-i\nu}{2} 
  \end{pmatrix}.
\]
A special case is a vertical Heisenberg translation~$T_{(0,\nu)}$ by $(0,\nu)\in\calN$ given by
\[(\zeta,\omega)\mapsto(\zeta,\omega+\nu).\]
%corresponds to the following element in $\PU(2,1)$
%\[\begin{pmatrix} 1 & 0 & 0 \\ 0 & 1+\frac{i\nu}{2} & \frac{i\nu}{2}\\ 0 & -\frac{i\nu}{2} & 1-\frac{i\nu}{2}\end{pmatrix}.\]
A {\defit Heisenberg rotation\/}~$R_{\mu}$ by $\mu\in\c$, $|\mu|=1$, is given by
\[(\zeta,\omega)\mapsto(\mu\cdot\zeta,\omega)\]
and corresponds to the following element in $\PU(2,1)$
\[
  \begin{pmatrix}
    \mu & 0 & 0 \\
    0 & 1 & 0\\ 
    0 & 0 & 1
  \end{pmatrix}.
\]
A minimal complex reflection~$\iota_{C_{\varphi}}$ of order~$n$ in a vertical chain~$C_{\varphi}$ with polar vector 
\[c_{\varphi}=\begin{bmatrix}1\\ -\bar{\varphi}\\ \bar{\varphi}\end{bmatrix}\]
 is given by
\[
  (\zeta,\omega)
  \mapsto
  \left(\mu\zeta+(1-\mu)\varphi,\omega-2|\varphi|^2\Im(1-\mu)+2\Im((1-\mu)\bar{\varphi}\zeta)\right)
\]
and corresponds to the following element in~$\PU(2,1)$
\[
  \begin{pmatrix}
    -\mu & -(1-\mu)\varphi & -(1-\mu)\varphi\\ 
  -(1-\mu)\bar{\varphi} & (1-\mu)|\varphi|^2-1 & (1-\mu)|\varphi|^2\\ 
  (1-\mu)\bar{\varphi} & -(1-\mu)|\varphi|^2 & -(1-\mu)|\varphi|^2-1
  \end{pmatrix},
\] 
where $\mu=\exp(2\pi i/n)$.
The complex reflection $\iota_{C_{\varphi}}$ can be decomposed as a product of a Heisenberg translation and a Heisenberg rotation:
\[\iota_{C_{\varphi}}=R_{\mu}\circ T_{(\xi,\nu)}=T_{(\mu\xi,\nu)}\circ R_{\mu},\]
where 
\[
  \xi=(\bar{\mu}-1)\varphi
  \quad\text{and}\quad
  \nu=-2|\varphi|^2\cdot\Im(1-\mu)=2|\varphi|^2\sin(2\pi/n).
\]
Heisenberg translations, Heisenberg rotations and complex reflections are isometries with respect to the Cygan metric.
The group of all Heisenberg translations is isomorphic to~$\calN$.
The group of all Heisenberg rotations $\{R_{\mu}\st\mu\in\c,~|\mu|=1\}$ is isomorphic to~$\U(1)$.
The group of their products $\calN\semiprod\U(1)$ contains all complex reflections.

\subsection{Products of reflections in chains:}
What effect does the minimal complex reflection of order~$n$ in the vertical chain~$C_\zeta$ have on another vertical chain, $C_\xi$, which intersects $\c\times\{0\}$ at~$\xi$?
\\
\linebreak
We calculate 
\begin{align*}
  \begin{pmatrix}
    -\mu & -(1-\mu)\zeta & -(1-\mu)\zeta \\
    -(1-\mu)\bar{\zeta} & (1-\mu)|\zeta|^2-1 & (1-\mu)|\zeta|^2 \\
    (1-\mu)\bar{\zeta} & -(1-\mu)|\zeta|^2 & -(1-\mu)|\zeta|^2-1
  \end{pmatrix}
  \begin{bmatrix}1\\ -\bar{\xi}\\ \bar{\xi}\end{bmatrix}
%  &=
%  \begin{pmatrix}
%    -\mu+(1-\mu)\zeta\bar{\xi}-(1-\mu)\zeta\bar{\xi} \\
%   -(1-\mu)\bar{\zeta}-((1-\mu)|\zeta|^2-1)\bar{\xi} + (1-\mu)\bar{\xi}|\zeta|^2 \\
%    (1-\mu)\bar{\zeta}+(1-\mu)\bar{\xi}|\zeta|^2 - ((1-\mu)|\zeta|^2+1)\bar{\xi}
%  \end{pmatrix}
  =
  \begin{bmatrix} -\mu\\ -(1-\mu)\bar{\zeta}+\bar{\xi}\\ (1-\mu)\bar{\zeta} -\bar{\xi}\end{bmatrix}.
\end{align*}
This vector is a multiple of
\[
  \begin{bmatrix} 
    1\\ 
    (1-\mu)\bar{\mu}\bar{\zeta}-\bar{\mu}\bar{\xi}\\
    -(1-\mu)\bar{\mu}\bar{\zeta} +\bar{\mu}\bar{\xi}
\end{bmatrix}
=\begin{bmatrix} 1 \\ -\overline{\left(\mu\xi-(\mu-1)\zeta\right)} \\ \overline{\left(\mu\xi-(\mu-1)\zeta\right)} \end{bmatrix}
\]
which is the polar vector of the vertical chain that intersects $\c\times\{0\}$ at $\mu\xi-(\mu-1)\zeta$. 
This corresponds to rotating $\xi$ around $\zeta$ through $\frac{2\pi}{n}$.
So if we have a vertical chain $C_{\xi}$, the minimal complex reflection of order~$n$ in another vertical chain~$C_{\zeta}$ rotates $C_{\xi}$ as a set around $C_{\zeta}$ through $\frac{2\pi}{n}$ (but not point-wise).
%as there is also vertical translation on the chain.

\subsection{Bisectors and spinal spheres:}
Unlike in the real hyperbolic space, there are no totally geodesic real hypersurfaces in $\chp$.
An acceptable substitute are the metric bisectors.
Let $z_1, z_2\in\chp$ be two distinct points.
The {\defit bisector equidistant\/} from~$z_1$ and~$z_2$ is defined as
\[\{z\in\chp\st \rho(z_1,z)=\rho(z_2,z)\}.\]
The intersection of a bisector with the boundary of~$\chp$ is a smooth hypersurface in~$\bchp$ called a {\defit spinal sphere\/}, which is diffeomorphic to a sphere. 
An example is the bisector
\[\calC=\{[z:it:1]\in\chp\st |z|^2<1-t^2,~z\in\c,~t\in\r\}.\]
Its boundary, the {\defit unit spinal sphere\/}, can be described as
\[U=\{(\zeta,\nu)\in\calN\st |\zeta|^4+\nu^2=1\}.\]

\section{Parametrisation of complex hyperbolic triangle groups of type $[m_1, m_2, 0;n_1,n_2,n_3]$}
\label{sec-param}
\noindent
For $r_1, r_2 \ge1$ and $\al\in(0,2\pi)$, let $C_1$, $C_2$ and $C_3$ be the complex geodesics with respective polar vectors
\[
  c_1 = \begin{bmatrix}1 \\ -r_2e^{-i\theta} \\ r_2e^{-i\theta} \end{bmatrix},\quad
  c_2 = \begin{bmatrix}1 \\ r_1e^{i\theta} \\ -r_1e^{i\theta} \end{bmatrix}
  \quad\mbox{and}\quad
  c_3 = \begin{bmatrix}0 \\ 1 \\ 0 \end{bmatrix},
\]
where $\theta=(\pi-\al)/2\in(-\pi/2,\pi/2)$.
The type of triangle formed by $C_1,C_2,C_3$ is an ultra-parallel $[m_1, m_2,0]$-triangle with angular invariant $\al$,
where $r_k=\cosh(m_k/2)$ for~$k=1,2$.
\\
\linebreak
Let $\iota_k$ be the minimal complex reflection of order $n_k$ in the chain $C_k$ for $k=1,2,3$.
The group $\Ga=\<\iota_1,\iota_2,\iota_3\>$ generated by these three complex reflections
is an ultra-parallel complex hyperbolic triangle group of type $[m_1, m_2, 0;n_1,n_2,n_3]$.
Looking at the arrangement of the chains $C_1$, $C_2$ and $C_3$ in the Heisenberg space $\calN$,
the finite chain~$C_3$ is the (Euclidean) unit circle in $\c\times\{0\}$,
whereas $C_1$ and~$C_2$ are vertical lines
through the points $\varphi_1 = r_2e^{i\theta}$ and $\varphi_2=-r_1e^{-i\theta}$ respectively, see Figure~\ref{fig-chains}.
For~$k=1,2$, the reflection~$\iota_k$ rotates any vertical chain as a set  through $\frac{2\pi}{n_k}$ around~$C_k$.

% FIGURE
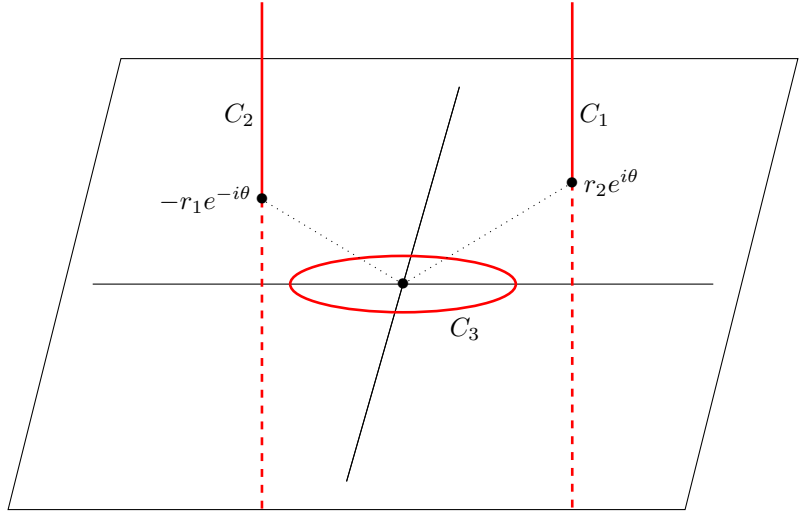
\begin{figure}[h]
\begin{center}
\begin{tikzpicture}[scale=0.75]
%  \setcoordinatesystem units <0.75cm,0.75cm> point at 0 0
  \path[clip] (-8,-5)--(-8,5)--(8,5)--(8,-5)--(-8,-5)--cycle;
  \draw (-7,-4)--(5,-4)--(7,4)--(-5,4)--(-7,-4)--cycle;
  \draw (-5.5,0)--(5.5,0);
  \draw (-1,-3.5)--(1,3.5);
  \draw (-1,-3.5)--(1,3.5);
  \draw[line width=.35mm, red] (-2.5,1.5)--(-2.5,5);
  \draw[line width=.35mm, red] (3,1.8)--(3,5);
  \draw[line width=.35mm, red,dashed] (-2.5,1.5)--(-2.5,-4);
  \draw[line width=.35mm, red,dashed] (3,1.8)--(3,-4);
  \draw[line width=.35mm, red] (0,0) ellipse (2cm and 0.5cm);
  %\ellipticalarc axes ratio 4:1 360 degrees from 2 0 center at 0 0
  \path[draw,dotted] (0,0)--(-2.5,1.5);
  \path[draw,dotted] (0,0)--(3,1.8);
  \foreach \Point in {(0,0), (-2.5,1.5), (3,1.8)}{\node at \Point {$\bullet$};}
  \node at (1.1,-0.8) {$C_3$};
  \node at (-2.9,3) {$C_2$};
  \node at (3.4,3) {$C_1$};
  \node at (3.7,1.8) {$r_2 e^{i\theta}$};
  \node at (-3.5,1.5) {$-r_1 e^{-i\theta}$};
\end{tikzpicture}
\end{center}
\caption{Chains $C_1$, $C_2$ and $C_3$ (figure from~\cite{andyanna}).}
\label{fig-chains}
\end{figure}

\noindent
In this case, the descriptions of complex reflections in section~\ref{heisenberg-isometries}  are as follows:
The reflection $\iota_k$ for $k=1,2$ is given by 
\[(\zeta,\omega) \mapsto \left(\mu_k\zeta+(1-\mu_k)\varphi_k,\omega+2|\varphi_k|^2\Im(1-\mu_k)+2\Im((1-\mu_k)\bar{\varphi_k}\zeta)\right),\]
where $\mu_k=\exp(2\pi i/n_k)$, and can be decomposed into a product of a Heisenberg translation and a Heisenberg rotation:
\[\iota_k=R_{\mu_k}\circ T_{(\xi_k,\nu_k)}=T_{(\mu_k\xi_k,\nu_k)}\circ R_{\mu_k},\]
where 
\[
  \xi_k=(\bar{\mu_k}-1)\varphi_k
  \quad\text{and}\quad
  \nu_k=-2|\varphi_k|^2\cdot\Im(1-\mu_k)=2|\varphi_k|^2\sin(2\pi/n_k).
\]
For $k=1,2$, the reflection~$\iota_k$ rotates any vertical chain as a set through $\frac{2\pi}{n_k}$ around~$C_k$.

%\begin{figure}[H]
%\centering
%\includegraphics[width=1\textwidth]{chains}
%\caption{Chains \(C_1, C_2\) and \(C_3\).}
%\end{figure}
%\noindent
%The reflections $\iota_k$ for $k=1,2$ are given by 
%\[
%  (\zeta,\omega)
%  \mapsto
%  (\mu_k\zeta+(1-\mu_k)\varphi_k,\omega+2|\varphi_k|^2\Im(1-\mu_k)+2\Im((1-\mu_k)\bar{\varphi_k}\zeta),
%\]
%where $\mu_k=\exp(2\pi i/n_k)$.
%and correspond to the following elements in~$\PU(2,1)$
%\[
%  \iota_k=
%  \begin{pmatrix}
%    -\mu_k & -(1-\mu_k)\varphi_k & -(1-\mu_k)\varphi_k \\
%    -(1-\mu_k)\bar{\varphi}_k & (1-\mu_k)|\varphi_k|^2-1 & (1-\mu_k)|\varphi_k|^2  \\
%    (1-\mu_k)\bar{\varphi}_k & -(1-\mu_k)|\varphi_k|^2 & -(1-\mu_k)|\varphi_k|^2-1 
%  \end{pmatrix}
%\]
%In particular, for the minimal complex reflection $\iota_{C_{\varphi}}$ of order~$3$ in a vertical chain $C_{\varphi}$,
%the corresponding element in $\PU(2,1)$ is
%\[
%  M
%  =
%  \begin{bmatrix}
%    \frac{1-i\sqrt{3}}{2}&-\frac{3-i\sqrt{3}}{2}\varphi&-\frac{3-i\sqrt{3}}{2}\varphi \\ 
%    -\frac{3-i\sqrt{3}}{2}\bar{\varphi} & \frac{3-i\sqrt{3}}{2}|\varphi|^2-1 & \frac{3-i\sqrt{3}}{2}|\varphi|^2 \\
%    \frac{3-i\sqrt{3}}{2}\bar{\varphi} & -\frac{3-i\sqrt{3}}{2}|\varphi|^2 & -\frac{3-i\sqrt{3}}{2}|\varphi|^2-1
%  \end{bmatrix},
%\]
%since $\mu=\exp(2\pi i/3)=-\frac{1-i\sqrt{3}}{2}$.
%For $n_3=2$ the complex reflection $\iota_3$ in~$C_3$ is given by 
%\[\iota_3([z_1:z_2:z_3])=[-z_1:z_2:-z_3]=[z_1:-z_2:z_3]\]
%and corresponds to the following element in~$\PU(2,1)$
%\[
%  \iota_3=
%  \begin{pmatrix}
%    -1 & 0 & 0 \\
%    0 & 1 & 0  \\
%    0 & 0 & -1 
%  \end{pmatrix}
%\]

\section{Orders of Reflection}
\label{ordreflc}
\noindent
The group \(\Gamma=\langle{\iota_1, \iota_2, \iota_3\rangle}\) generated by the three complex reflections is an ultra-parallel complex hyperbolic triangle group of type \([m_1, m_2, 0; n_1, n_2, n_3]\). We want to know for what orders of the complex reflections \(\iota_1\) and \(\iota_2\) is the group \(\Gamma\) discrete. We will use the work of Hersonsky and Paulin \cite{hersonpaul}.
\\
\linebreak
Recall the Heisenberg group $\calN$ endowed with the group law
\[(\zeta_1, \nu_1) \ast (\zeta_2, \nu_2) = \left(\zeta_1+\zeta_2, \nu_1+\nu_2+2\Im\left(\zeta_1\bar{\zeta_2}\right)\right)\]
introduced in section \ref{heisenberg-group}. First, we will use Proposition \(5.4\) (specifically when \(n=2\)) of \cite{hersonpaul}:
\begin{proposition}
\label{herpaul}
Let \(\Gamma\) be a discrete cocompact subgroup in $\calN$.
Let $\pi:\calN\to\mathbb{C}$ be the canonical projection defined by $\pi(\zeta,\nu)=\zeta$.
Then $\pi(\Gamma)$ is a cocompact lattice in $\mathbb{C}$.
\end{proposition}

\begin{proof}
From the Heisenberg group $\calN$ there exists a central extension
\[0\rightarrow\mathbb{R}\rightarrow\calN\rightarrow\mathbb{C} \rightarrow 0.\]
By this, we have that 
\[
  \ker(\pi)=\mathbb{R}
  \quad\mbox{and}\quad
  \calN/\mathbb{R}=\mathbb{C}.
\]
Note that $\Gamma\cap\mathbb{R}$ is a normal subgroup of $\Gamma$, since $\mathbb{R}$ is in the centre of $\calN$.
Therefore, the group 
\[G=\Gamma/\left(\Gamma \cap\mathbb{R}\right)\]
which identifies to $\pi(\Gamma)$ acts on $\mathbb{C}$.
We can see that $G$ acts with bounded quotient on $\calN/\mathbb{R}$ and therefore $\pi(\Gamma)$ acts cocompactly on $\mathbb{C}$. 
\\
\linebreak
We next need to show that $G$ acts discretely on $\mathbb{C}$.
For a contradiction, suppose not.
For a sequence to converge on the plane $\mathbb{C}$,
we are able to bound their corresponding elements in $\Gamma$ by applying a vertical Heisenberg translation $H$ in $\Gamma$,
which exists due to $\Gamma$ being non-abelian.
This implies that there is a convergent subsequence to an element \((0,t), t\in\mathbb{R}\).
This is the required contradiction since $\Gamma$ is discrete, and hence discrete on $\mathbb{C}$, which completes the proof. 
\end{proof}

\noindent
We now use two more results from Hersonsky and Paulin \cite{hersonpaul} (Theorem \(5.3\) and Proposition \(5.8\)):
\begin{theorem}
\label{indexconstant}
Let \(G\) be a cocompact, discrete, torsion-free subgroup of isometries of \(\calN_{2n-1}\).
There exists a universal constant \(I_n\) such that \(G\) contains a cocompact lattice of index less than or equal to \(I_n\).
Moreover,
\[I_n \leq 2\left(6\pi\right)^{\frac{n}{2}(n-1)}.\]
\end{theorem}

\begin{proposition}
\label{index}
We have \(I_2=6\).
\end{proposition}

\noindent
For \(n=2\), Theorem \ref{indexconstant} and Proposition \ref{index} tells us that there exists a constant, \(I_2\), such that the group of Heisenberg translations, which is a cocompact, discrete, torsion free subgroup of the group of Heisenberg isometries, contains a lattice of index less than or equal to \(I_2=6\). From Proposition \ref{herpaul}, we know that the canonical projection of this isometry group is a cocompact lattice in \(\mathbb{C}\). 
\\
\linebreak
That is, by the classification of Euclidean
$2$-orbifolds, the group of Heisenberg translations contains a lattice subgroup of index $m=1,2,3,4,6$. Moreover, if $m=2$ then the canonical projection of the translation subgroup is $\calS^2(2,2,2,2)$ or $\calP^2(2,2)$.
If $m=3,4,6$ then the canonical projection of the translation subgroup is a $(3,3,3)-$, $(2,4,4)-$, or $(2,3,6)-$triangle group respectively.
As the Heisenberg translations are generated by the complex reflections, the orders of the complex reflections have to be contained in one of these groups.
This gives rise to Theorem \ref{reforders}.

\begin{remark}
The cases when the unordered pair of orders of the complex reflections \(\iota_1\) and \(\iota_2\) is \(\lbrace{2,2\rbrace}\) and \(\lbrace{3,3\rbrace}\) were discussed in \cite{andyanna} and \cite{povprat} respectively. 
\end{remark}

\section{Compression Property}
\label{sec-compression}
\noindent
Let $C_1,C_2,C_3$ be chains in~$\calN$ as in the previous section.
Let $\iota_k$ be the minimal complex reflection of order~$n_k$ in the chain~$C_k$ for $k=1,2,3$.
We will assume that $n_3=2$.
To prove the discreteness of the group $\<\iota_1,\iota_2,\iota_3\>$ we will use the following version of Klein's combination theorem discussed in \cite{WyssGall}:

\begin{proposition}

\label{criterion}

If there exist subsets $U_1$, $U_2$ and~$V$ in~$\calN$ with $U_1\cap U_2=\varnothing$ and $V\subsetneq U_1$ such that 
$\iota_3(U_1) = U_2$ and  $g(U_2)\subsetneq V$ for all~$g\ne\Id$ in $\<\iota_1,\iota_2\>$,
then the group $\<\iota_1,\iota_2,\iota_3\>$ is a discrete subgroup of $\PU(2,1)$.
Groups with such properties are called {\defit compressing\/}.
\end{proposition}

\noindent
Projecting the actions of complex reflections~$\iota_1$ and~$\iota_2$ to~$\c\times\{0\}$ we obtain
rotations~$j_1$ and $j_2$ of~$\c$ around $\varphi_1=r_2e^{i\theta}$ and $\varphi_2=-r_1e^{-i\theta}$
through $\frac{2\pi}{n_1}$ and $\frac{2\pi}{n_2}$ respectively.
We will use Proposition~\ref{criterion} to prove the following Lemma:

%\EE=$\<\iota_1,\iota_2\>$
%$\Lambda=\<j_1,j_2\>$

\begin{lemma}
\label{f(0)}
If $|f(0)|\ge2$ for all $f\ne\Id$ in $\<j_1,j_2\>$ and $|h(0)|\ge2$ for all vertical Heisenberg translations $h\ne\Id$ in $\<\iota_1,\iota_2\>$,
then the group $\<\iota_1,\iota_2,\iota_3\>$ is discrete.
\end{lemma}

\begin{proof}
Consider the unit spinal sphere 
\[U=\{(\zeta, \nu)\in\calN\st|\zeta|^4+\nu^2=1\}.\]
The complex reflection~$\iota_3$ in~$C_3$ is given by
\[\iota_3([z_1:z_2:z_3])=[-z_1:z_2:-z_3]=[z_1:-z_2:z_3].\]
The complex reflection~$\iota_3$ preserves the bisector
\[\calC=\{[z:it:1]\in\chp\st |z|^2<1-t^2, z\in\c, t\in\r\}\]
and hence preserves the unit spinal sphere~$U$ which is the boundary of the bisector~$\calC$.
The complex reflection~$\iota_3$ interchanges the points $[0:1:1]$ and $[0:-1:1]$ in~$\chp$,
which correspond to the points $(0,0)$ and $\infty$ in $\calN$.
Therefore, $\iota_3$ leaves $U$ invariant and switches the inside of $U$ with the outside.
\\
\linebreak
Let $U_1$ be the part of $\calN\backslash U$ outside~$U$, containing~$\infty$,
and let $U_2$ be the part inside~$U$, containing the origin.
Clearly
\[U_1\cap U_2=\varnothing\quad\text{and}\quad \iota_3(U_1)=U_2.\]
Therefore, if we find a subset $V\subsetneq U_1$ such that $g(U_2)\subsetneq V$
for all elements $g\ne\Id$ in $\<\iota_1,\iota_2\>$,
then we will show that $\<\iota_1,\iota_2,\iota_3\>$ is discrete.
Let
\[W=\{(\zeta,\nu)\in\calN\st|\zeta|=1\}\]
be the cylinder consisting of all vertical chains through $\zeta\in\c$ with $|\zeta|=1$.
Let
\[
  W_1=\{(\zeta,\nu)\in\calN\st|\zeta|>1\}
  \quad\text{and}\quad
  W_2=\{(\zeta,\nu)\in\calN\st|\zeta|<1\}
\]
be the parts of $\calN\backslash W$ outside and inside the cylinder~$W$ respectively. 
We have $U_2\subset W_2$ and so $g(U_2)\subset g(W_2)$ for all~$g\in\<\iota_1,\iota_2\>$.
The set $W_2$ is a union of vertical chains.
We know that elements of $\<\iota_1,\iota_2\>$ map vertical chains to vertical chains.
There is also a vertical translation on the chain itself.
Therefore, we look at both the intersection of the images of~$W_2$ with $\c\times\{0\}$ and the vertical displacement of~$W_2$. 
\\
\linebreak
Elements of $\<\iota_1,\iota_2\>$ move the intersection of $W_2$ with $\c\times\{0\}$ by rotations $j_1$ and $j_2$
around \(r_2e^{i\theta}\) and \(-r_1e^{-i\theta}\) through $\frac{2\pi}{n_1}$ and $\frac{2\pi}{n_2}$ respectively.
Provided that the interior of the unit circle is mapped completely off itself under all non-identity elements in $\<j_1,j_2\>$,
then the same is true for $W_2$ and hence for $U_2$ under all elements in $\<\iota_1,\iota_2\>$ that are not vertical Heisenberg translations.
\\
\linebreak
A vertical Heisenberg translation will shift~$W_2$ and its images $g(W_2)$ vertically by the same distance,
hence the same is true for $U_2$ and its images $g(U_2)$.
\\
\linebreak
We choose $V$ to be the union of all the images of $U_2$ under all non-vertical elements of $\<\iota_1,\iota_2\>$.
This subset will satisfy the compressing conditions
assuming that the interior of the unit circle is mapped off itself by any non-identity element in $\<j_1,j_2\>$
and that the interior of the unit spinal sphere $U$ is mapped off itself
by any non-identity vertical Heisenberg translation in $\<\iota_1,\iota_2\>$.
Since the radius of the unit circle is preserved under rotations,
we need to show that the origin is moved the distance of at least twice the radius of the circle:
\[|f(0)|\ge2\quad\text{for all}~f\in\<j_1,j_2\>,~f\ne\Id.\]
Since vertical translations shift the spinal spheres vertically,
we need to show that they shift by at least the height of the spinal sphere:
\[|h(0)|\ge2\quad\text{for all vertical Heisenberg translations}~h\in\<\iota_1,\iota_2\>,~h\ne\Id.\]
We see that the conditions of this Lemma ensure that the sets $U_1$, $U_2$ and~$V$ satisfy the conditions of Proposition~\ref{criterion}.
\end{proof}

%\section{Parametrisation of complex hyperbolic %triangle groups of type $[m, m, 0; n_1,n_2,2]$}

%\noindent
%We will now focus on the case of %$[m_1,m_2,0;n_1,n_2,n_3]$-groups with %m_1=m_2=m$, $n_3=2$ and $\{n_1,n_2\}=\{2,3\},\{2,4\},\{4,4\}, \{2,6\},\{3,6\}$.
%In these cases the setting described in %section~\ref{sec-param} is as follows.
%We consider the following configuration of %chains in~$\calN$: $C_3$ is the (Euclidean) unit circle in $\c\times\{0\}$, whereas $C_1$ and~$C_2$ are vertical lines through the points $\varphi_1 = re^{i\theta}$ and $\varphi_2=-re^{-i\theta}$ respectively, where $r=\cosh(m/2)$ and $\theta\in(-\pi/2,\pi/2)$.
%The type of triangle formed by $C_1,C_2,C_3$ is an ultra-parallel $[m,m,0]$-triangle with angular invariant $\al=\pi-2\theta\in(0,2\pi)$.
%We will consider the ultra-parallel triangle group $\Ga=\<\iota_1,\iota_2,\iota_3\>$ generated by the minimal complex reflections $\iota_1,\iota_2,\iota_3$ of orders~$n_1,n_2,2$ in the chains~$C_1,C_2,C_3$ respectively.

\section{Subgroup of Heisenberg Translations}
\label{heitran}

\noindent
Let $\Ga=\<\iota_1,\iota_2,\iota_3\>$ be as in section \ref{sec-param} with~$n_3=2$.
In this section we will consider the structure of the subgroup $\calE=\<\iota_1,\iota_2\>$ in more detail.
Projecting the action of the complex reflection~$\iota_k$ to~$\c\times\{0\}$ we obtain
a rotation~$j_k$ of~$\c$ of order~$n_k$.
The action of~$\calE$ projects to the action of the group~$\Lambda$ generated by two rotations of orders~$n_1$ and~$n_2$ respectively.
In the cases considered here, $\Lambda$ is the triangle group of signature~$(3,3,3)$, ~$(2,3,6)$ or~$(2,4,4)$.
We will consider the subgroup~$\calT$ that consists of all Heisenberg translations in~$\calE$.
In each of the cases, we will describe a system of generators~$T_1,T_2,H$ of~$\calT$,
where $T_1$ and~$T_2$ project to translations of the shortest length in~$\Lambda$ and $H$ is vertical.
For the purpose of the following calculations, we will use the notation
$\iota_{a_1a_2...a_n}=\iota_{a_1}\iota_{a_2}...\iota_{a_n}$.

\begin{proposition}
\label{structure-E}
Let $n_k$ be the order of the minimal complex reflection $\iota_k$ for~$k=1,2$.
%Let $R=\sqrt{r_1^2+r_2^2+2r_1r_2\cos(2\theta)}$.
% $r_1=r_2=r\Longrightarrow R=2r\cos(\theta)$.
Let $R=2r\cos(\theta)$.
Let the Heisenberg translations~$T_1=T_{(v_1,t_1)}$ and $T_2=T_{(v_2,t_2)}$ be defined by 
\begin{align*}
&T_1=\iota_{212},~T_2=\iota_{112}
~\text{for}~(n_1,n_2)=(3,3),\\
&T_1=\iota_{21212},~T_2=\iota_{12212}
~\text{for}~(n_1,n_2)=(2,3),\\
&T_1=\iota_{2122},~T_2=\iota_{2212}
~\text{for}~(n_1,n_2)=(2,6),\\
&T_1=\iota_{1122},~T_2=\iota_{12222}
~\text{for}~(n_1,n_2)=(3,6),\\ % T_2=\iota_{2211}^{-1}
&T_1=\iota_{212},~T_2=\iota_{122}
~\text{for}~(n_1,n_2)=(2,4),\\
&T_1=\iota_{1112},~T_2=\iota_{2111}
~\text{for}~(n_1,n_2)=(4,4).
\end{align*}
Let the Heisenberg translation~$H=T_{(0,\nu)}$ be defined as follows:
If $(p,q,r)$ is a permutation of one of the triples~$(3,3,3)$, $(2,3,6)$, $(2,4,4)$ and~$(n_1,n_2)=(p,q)$ then
\[H=(\iota_{12})^r.\]
%\begin{align*}
%&H=(\iota_{12})^{3/a}~\text{for}~(n_1,n_2)=(3,3),\\
%&H=(\iota_{12})^{6/a}~\text{for}~(n_1,n_2)=(2,3),(2,6),(3,6),\\
%&H=(\iota_{12})^{4/a}~\text{for}~(n_1,n_2)=(2,4),(4,4).
%\end{align*}
Then we have
\begin{enumerate}[(a)]
\item
\begin{enumerate}[(1)]
\item
For~$(n_1,n_2)=(3,3)$,
\begin{align*}
&v_1=iR\sqrt{3},~
v_2=e^{\pi i/6}R\sqrt{3},~
|v_1|=|v_2|=R\sqrt{3},~
\nu=6R^2\sqrt{3}.
\end{align*}
\item
For~$(n_1,n_2)=(2,3)$,
\begin{align*}
&v_1=i2R\sqrt{3},~
v_2=e^{\pi i/6}2R\sqrt{3},~
|v_1|=|v_2|=2R\sqrt{3},~
\nu=24R^2\sqrt{3}.
\end{align*}
\item
For~$(n_1,n_2)=(2,6)$,
\begin{align*}
&v_1=e^{\pi i/3}2R,~
v_2=e^{2\pi i/3}2R,~
|v_1|=|v_2|=2R,~
\nu=4R^2\sqrt{3}.
\end{align*}
\item
For~$(n_1,n_2)=(3,6)$,
\begin{align*}
&v_1=e^{\pi i/6}R\sqrt{3},~
v_2=e^{-\pi i/6}R\sqrt{3},~
|v_1|=|v_2|=R\sqrt{3},~
\nu=2R^2\sqrt{3}.
\end{align*}
\item
For~$(n_1,n_2)=(2,4)$:
\begin{align*}
&v_1=i2R,~
v_2=2R,~
|v_1|=|v_2|=2R,~
\nu=16R^2.
\end{align*}
\item
For~$(n_1,n_2)=(4,4)$:
\begin{align*}
&v_1=e^{\pi i/4}R\sqrt{2},~
v_2=e^{3\pi i/4}R\sqrt{2},~
|v_1|=|v_2|=R\sqrt{2},~
\nu=4R^2.
\end{align*}
\end{enumerate}
\item
In all cases, $[T_1,T_2]=(0,\pm\,a\cdot\nu)=H^{\pm a}$,
where
\begin{align*}
&a=1~\text{for}~(n_1,n_2)=(3,3),(2,3),(2,4),\\
&a=2~\text{for}~(n_1,n_2)=(4,4),(2,6),\\
&a=3~\text{for}~(n_1,n_2)=(3,6)
\end{align*}
and
\begin{align*}
  \nu&=\frac{32R^2\sqrt{3}}{\gcd(n_1,n_2)}
  \sin^2\left(\frac{\pi}{n_1}\right)
  \sin^2\left(\frac{\pi}{n_2}\right)
  ~\text{for}~(n_1,n_2)=(3,3),(2,3),(2,6),(3,6),\\
  \nu&=\frac{64R^2}{\gcd(n_1,n_2)}
  \sin^2\left(\frac{\pi}{n_1}\right)
  \sin^2\left(\frac{\pi}{n_2}\right)
  ~\text{for}~(n_1,n_2)=(2,4),(4,4).
\end{align*}
\item
The group~$\calT$ of all Heisenberg translations in~$\calE$
is a normal subgroup of~$\calE$ generated by~$T_1$, $T_2$ and~$H$.
\item
The quotient group~$\calE/\calT$ is cyclic:
\begin{align*}
&\calE/\calT=\<\iota_{12}\>\quad\text{for}~(n_1,n_2)=(2,3),\\
&\calE/\calT=\<\iota_{2}\>
\quad\text{for}~(n_1,n_2)=(3,3),(2,6),(3,6),(2,4),(4,4).
\end{align*}
We will use the following sets of representatives~$\Si$ of elements of~$\calE/\calT$:
\begin{align*}
&\Si=\{1,\iota_1,\iota_{11}\}
~\text{for}~(n_1,n_2)=(3,3),\\
&\Si=\{1,\iota_1,\iota_2,\iota_{12},\iota_{22},\iota_{122}\}
~\text{for}~(n_1,n_2)=(2,3),\\
&\Si=\{1,\iota_1,\iota_2,\iota_{21},\iota_{22},\iota_{221}\}
~\text{for}~(n_1,n_2)=(2,6),\\ 
&\Si=\{1,\iota_2,\iota_{22},\iota_{112},\iota_{221},\iota_{222}\}
~\text{for}~(n_1,n_2)=(3,6),\\
&\Si=\{1,\iota_1,\iota_2,\iota_{21}\}~\text{for}~(n_1,n_2)=(2,4),\\
&\Si=\{1,\iota_1,\iota_{11},\iota_{111}\}
~\text{for}~(n_1,n_2)=(4,4).
\end{align*}
\item
$T_1,T_2,H$ satisfy the relations
\[[H,T_1]=[H,T_2]=1,\quad [T_1,T_2]=H^{\pm a}.\]
Every element of~$\calT$ is of the form $T_1^xT_2^yH^n$ for some $x,y,n\in\mathbb{Z}$.
\item
The shortest non-trivial vertical translations in~$\calE$ are~$H^{\pm1}$.
\item
Let $w_1=\frac{1}{2}(v_2-v_1)$ and $w_2=\frac{1}{2}(v_1+v_2)$.
\begin{enumerate}[(1)]
\item
For~$(n_1,n_2)=(3,3)$,
\begin{align*}
&w_1=(\sqrt{3}-i)R\cdot\sqrt{3}/4,~
w_2=(i\sqrt{3}+1)R\cdot 3/4,\\
&|w_1|=R\sqrt{3}/2,~|w_2|=3R/2,~
w_2=i\sqrt{3}\cdot w_1.
\end{align*}
\item
For~$(n_1,n_2)=(2,3)$,
\begin{align*}
&w_1=(\sqrt{3}-i)R\cdot\sqrt{3}/2,~
w_2=(i\sqrt{3}+1)R\cdot 3/2,\\
&|w_1|=R\sqrt{3},~|w_2|=3R,~
w_2=i\sqrt{3}\cdot w_1.
\end{align*}
\item
For~$(n_1,n_2)=(2,6)$,
\begin{align*}
&w_1=-R,~
w_2=iR\sqrt{3},\\
&|w_1|=R,~|w_2|=R\sqrt{3},~
w_2=-i\sqrt{3}\cdot w_1.
\end{align*}
\item
For~$(n_1,n_2)=(3,6)$,
\begin{align*}
&w_1=-iR\cdot\sqrt{3}/2,~
w_2=3R/2,\\
&|w_1|=R\sqrt{3}/2,~|w_2|=3R/2,~
w_2=i\sqrt{3}\cdot w_1.
\end{align*}
\item
For~$(n_1,n_2)=(2,4)$:
\begin{align*}
&w_1=(1-i)\cdot R,~
w_2=(i+1)\cdot R,\\
&|w_1|=|w_2|=R\sqrt{2},~
w_2=i\cdot w_1.
\end{align*}
\item
For~$(n_1,n_2)=(4,4)$:
\begin{align*}
&w_1=-R,~
w_2=iR,\\
&|w_1|=|w_2|=R,~
w_2=-i\cdot w_1.
\end{align*}
\end{enumerate}
\end{enumerate}
\end{proposition}

\begin{remark}
The group~$\calT$ has the presentation
\[
  \calT
  =\<T_1,T_2,H\:|\:[T_1,T_2]=H^a,[H,T_1]=[H,T_2]=1\>
\]
and is isomorphic to the uniform lattice~$N_a$ as defined in section~\ref{heisenberg-group}.
\end{remark}

\begin{proof}
~~~
\begin{enumerate}[(a)]
\item
The expressions for~$v_1$, $v_2$ and~$\nu$ are found by direct computation using the matrices of the complex reflections~$\iota_1$ and~$\iota_2$ or the formula
\[R_{\mu}\circ T_{(\xi,\nu)}=T_{(\mu\xi,\nu)}\circ R_{\mu}.\]
\item
Using
\[[T_1,T_2]=[(v_1,t_1),(v_2,t_2)]=(0,4\Im(v_1\bar{v}_2)),\]
we calculate the commutator
\begin{align*}
&[T_1,T_2]=(0,6R^2\sqrt{3})=(0,\nu)=H~\text{for}~(n_1,n_2)=(3,3),\\
&[T_1,T_2]=(0,24R^2\sqrt{3})=(0,\nu)=H~\text{for}~(n_1,n_2)=(2,3),\\
&[T_1,T_2]=(0,-8R^2\sqrt{3})=(0,-2\nu)=H^{-2}~\text{for}~(n_1,n_2)=(2,6),\\
&[T_1,T_2]=(0,6R^2\sqrt{3})=(0,3\nu)=H^3~\text{for}~(n_1,n_2)=(3,6),\\
&[T_1,T_2]=(0,16R^2)=(0,\nu)=H~\text{for}~(n_1,n_2)=(2,4),\\
&[T_1,T_2]=(0,-8R^2)=(0,-2\nu)=H^{-2}~\text{for}~(n_1,n_2)=(4,4),
\end{align*}
hence in all cases
\[[T_1,T_2]=H^{\pm a}.\]
\item
Every element in~$\calE$ can be written
as a product of~$\iota_1$ and~$\iota_2$
using $\iota_k^{-1}=\iota_k^{n_k-1}$ for~$k=1,2$.
\begin{enumerate}[(1)]
\item
$(n_1,n_2)=(3,3)$:
Using
\begin{align*}
&\iota_{221}=T_2^{-1},~
\iota_{122}=T_2T_1^{-1},~
\iota_{211}=T_1T_2^{-1},\\
&\iota_{121}=T_2T_1^{-1}T_2^{-1},
\end{align*}
we can write every element in~$\calE$
as a product of an element of~$\<T_1,T_2\>$
and a word in~$\iota_1$ and~$\iota_2$
of length at most~$2$.
A word of length at most~$2$ is a Heisenberg translation
if and only if it is the identity,
hence $\calT=\<T_1,T_2\>$.
\item
$(n_1,n_2)=(2,3)$:
Using
\begin{align*}
&\iota_{12122}=T_2\cdot T_1^{-1},~
\iota_{12212}=T_2,~
\iota_{21212}=T_1,\\
&\iota_{21221}=T_1\cdot T_2^{-1},~
\iota_{22121}=T_2^{-1},\\
&\iota_{12121}=T_2\cdot T_1^{-1}\cdot\iota_{221},~ \iota_{22122}=T_1^{-1}\cdot\iota_{21},
\end{align*}
we can write every element in~$\calE$
as a product of an element of~$\<T_1,T_2\>$
and a word in~$\iota_1$ and~$\iota_2$
of length at most~$4$.
A word of length at most~$4$ is a Heisenberg translation
if and only if it is the identity,
hence $\calT=\<T_1,T_2\>$.
\item
$(n_1,n_2)=(2,6)$:
Using
\begin{align*}
&\iota_{1222}=H\cdot T_2^{-1}\cdot T_1,~
\iota_{2122}=T_1,~
\iota_{2212}=T_2,\\
&\iota_{2221}=T_1^{-1}\cdot T_2\cdot H^{-1},\\
&\iota_{1212}=H\cdot T_2^{-1}\cdot\iota_{22},~
\iota_{1221}=H\cdot T_2^{-1}\cdot T_1\cdot T_2^{-1}\cdot \iota_{22},\\
&\iota_{2121}=T_1\cdot T_2^{-1}\cdot\iota_{22},~
\iota_{2222}=T_1^{-1}\cdot \iota_{21},
\end{align*}
we can write every element in~$\calE$
as a product of an element of~$\<T_1,T_2,H\>$
and a word in~$\iota_1$ and~$\iota_2$
of length at most~$3$.
A word of length at most~$3$ is a Heisenberg translation
if and only if it is the identity,
hence $\calT=\<T_1,T_2,H\>$.
\item
$(n_1,n_2)=(3,6)$:
Using
\begin{align*}
&\iota_{1122}=T_1,~
\iota_{1221}=T_2\cdot T_1^{-1},~
\iota_{2112}=T_1\cdot T_2^{-1}\cdot H^{2},\\
&\iota_{2121}=H,~
\iota_{2211}=T_2^{-1},\\
&\iota_{1121}=T_2\cdot H\cdot\iota_{2},~
\iota_{1211}=H^{-1}\cdot T_2\cdot T_1^{-1}\cdot \iota_{2},\\
&\iota_{1222}=T_2\cdot H^{-1}\cdot\iota_{121},~
\iota_{2122}=H\cdot\iota_{112},~
\iota_{2212}=T_2^{-1}\cdot \iota_{112},\\
&\iota_{2221}=T_2^{-1}\cdot H\cdot T_1^{-1}\cdot \iota_{112},~
\iota_{2222}=T_1^{-1}\cdot\iota_{11},
\end{align*}
we can write every element in~$\calE$
as a product of an element of~$\<T_1,T_2,H\>$
and a word in~$\iota_1$ and~$\iota_2$
of length at most~$3$.
A word of length at most~$3$ is a Heisenberg translation
if and only if it is the identity,
hence $\calT=\<T_1,T_2,H\>$.
\item
$(n_1,n_2)=(2,4)$:
Using
\begin{align*}
&\iota_{121}=T_2\cdot T_1^{-1}\cdot\iota_{2},~
\iota_{122}=T_2,~
\iota_{212}=T_1,~
\iota_{221}=T_2^{-1},\\
&\iota_{222}=T_1^{-1}\cdot\iota_{21},
\end{align*}
we can write every element in~$\calE$
as a product of an element of~$\<T_1,T_2\>$
and a word in~$\iota_1$ and~$\iota_2$
of length at most~$2$.
A word of length at most~$2$ is a Heisenberg translation
if and only if it is the identity,
hence $\calT=\<T_1,T_2\>$.
\item
$(n_1,n_2)=(4,4)$:
Using
\begin{align*}
&\iota_{1112}=T_1,~
\iota_{1121}=T_1\cdot H^{-1}\cdot T_2^{-1}\cdot T_1^{-1},~
\iota_{1122}=T_1\cdot H^{-1}\cdot T_2^{-1},\\
&\iota_{1211}=H\cdot T_1^{-1},~
\iota_{1212}=H,~
\iota_{1221}=T_2^{-1}\cdot T_1^{-1},~
\iota_{1222}=T_2^{-1},\\
&\iota_{2111}=T_2,~
\iota_{2112}=T_2\cdot T_1,~
\iota_{2121}=T_2\cdot T_1\cdot H^{-1}\cdot T_2^{-1}\cdot T_1^{-1},\\
&\iota_{2122}=T_2\cdot T_1\cdot H^{-1}\cdot T_2^{-1},~
\iota_{2211}=T_2\cdot H\cdot T_1^{-1},\\
&\iota_{2212}=T_2\cdot H,~
\iota_{2221}=T_1^{-1},
\end{align*}
we can write every element in~$\calE$
as a product of an element of~$\<T_1,T_2.H\>$
and a word in~$\iota_1$ and~$\iota_2$
of length at most~$3$.
A word of length at most~$3$ is a Heisenberg translation
if and only if it is the identity,
hence $\calT=\<T_1,T_2,H\>$.
\end{enumerate}
\item
\begin{enumerate}[(1)]
\item
$(n_1,n_2)=(3,3)$:
$
 \iota_1\cdot T_2=\iota_2
 \Longrightarrow
 [\iota_1]=[\iota_2]
$,
hence every element of $\calE/\calT$ is a power of the element~$[\iota_2]$ of order~$3$.
The elements of~$\calE/\calT$ can be represented by the elements of $\Si=\{1,\iota_1,\iota_{11}\}$
as 
\[
[\iota_2]=[\iota_1],~
[\iota_2]^2=[\iota_{11}].
\]
\item
$(n_1,n_2)=(2,3)$:
\begin{align*}
 &\iota_1\cdot T_1=\iota_{121212}
 \Longrightarrow
 [\iota_1]=[\iota_{12}]^3,\\
 &\iota_2\cdot T_2^{-1}T_1=\iota_{12121212}
 \Longrightarrow
 [\iota_2]=[\iota_{12}]^4,
\end{align*}
hence every element of $\calE/\calT$ is a power of the element~$[\iota_{12}]$ of order~$6$.
The elements of~$\calE/\calT$ can be represented by the elements of
$\Si=\{1,\iota_1,\iota_2,\iota_{12},\iota_{22},\iota_{122}\}$
as 
\[
[\iota_{12}]^2=[\iota_{22}],~
[\iota_{12}]^3=[\iota_1],~
[\iota_{12}]^4=[\iota_2],~
[\iota_{12}]^5=[\iota_{122}].
\]
\item
$(n_1,n_2)=(2,6)$:
$
 \iota_1\cdot HT_2^{-1}T_1=\iota_{222}
 \Longrightarrow
 [\iota_1]=[\iota_2]^3
$,
hence every element of $\calE/\calT$ is a power of the element~$[\iota_2]$ of order~$6$.
The elements of~$\calE/\calT$ can be represented by the elements of
$\Si=\{1,\iota_1,\iota_2,\iota_{21},\iota_{22},\iota_{221}\}$
as 
\[
[\iota_2]^2=[\iota_{22}],~
[\iota_2]^3=[\iota_1],~
[\iota_2]^4=[\iota_{21}],~
[\iota_2]^5=[\iota_{221}].
\]
\item 
$(n_1,n_2)=(3,6)$:
$
 T_2^{-1}\cdot\iota_{1}=\iota_{22}
 \Longrightarrow
 [\iota_1]=[\iota_2]^2
$,
hence every element of $\calE/\calT$ is a power of the element~$[\iota_2]$ of order~$6$.
The elements of~$\calE/\calT$ can be represented by the elements of
$\Si=\{1,\iota_2,\iota_{22},\iota_{112},\iota_{221},\iota_{222}\}$
as 
\[
[\iota_2]^2=[\iota_{22}],~
[\iota_2]^3=[\iota_{222}],~
[\iota_2]^4=[\iota_{221}],~
[\iota_2]^5=[\iota_{112}].
\]
\item
$(n_1,n_2)=(2,4)$:
$
 T_2^{-1}\cdot\iota_{1}=\iota_{22}
 \Longrightarrow
 [\iota_1]=[\iota_2]^2
$,
hence every element of $\calE/\calT$ is a power of the element~$[\iota_2]$ of order~$4$.
The elements of~$\calE/\calT$ can be represented by the elements of
$\Si=\{1,\iota_1,\iota_2,\iota_{21}\}$
as 
\[
[\iota_2]^2=[\iota_1],~
[\iota_2]^3=[\iota_{21}].
\]
\item
$(n_1,n_2)=(4,4)$:
$
 T_2\cdot\iota_{1}=\iota_{2}
 \Longrightarrow
 [\iota_1]=[\iota_2]
$,
hence every element of $\calE/\calT$ is a power of the element~$[\iota_2]$ of order~$4$.
The elements of~$\calE/\calT$ can be represented by the elements of $\Si=\{1,\iota_1,\iota_{11},\iota_{111}\}$
as 
\[
[\iota_2]=[\iota_1],~
[\iota_2]^2=[\iota_{11}],~
[\iota_2]^3=[\iota_{111}].
\]
\end{enumerate}
\item
Vertical translations are central in~$\calN$,
hence $[H,T_1]=[H,T_2]=1$.
Using the relations $HT_1=T_1H$, $HT_2=T_2H$ and~$T_1T_2=H^{\pm a}T_2T_1$,
we can rearrange every element in~$\calT=\<T_1,T_2,H\>$
into the form $T_1^xT_2^yH^n$ for some $x,y,n\in\mathbb{Z}$.
\item
For an element $T_1^xT_2^yH^n\in\calE$ to be vertical, we need $x=y=0$,
hence any vertical element of~$\calE$ is some power of~$H$.
\item
We can calculate~$w_1$ and~$w_2$
using~$v_1$ and~$v_2$ calculated in the previous parts.
\end{enumerate}
\end{proof}

\begin{remark}
An alternative approach to the understanding of the structure of the subgroup~$E=\<\iota_1,\iota_2\>$ is outlined in Appendix~\ref{Dekimpe} and uses the classification of almost-crystallographic groups by Dekimpe~\cite{dekimpe}.
\end{remark}

\section{Discreteness Results}
\label{disres}

\noindent
Let $\Ga=\<\iota_1,\iota_2,\iota_3\>$ be as in section~\ref{sec-param} with~$n_3=2$ and $m_1=m_2$, i.e.\ $r_1=r_2=r$.
We will now use Lemma~\ref{f(0)} to find conditions for the group~$\Ga$ to be discrete for all possible orders~$\{n_1,n_2\}$ of the complex reflections.

\begin{proposition}
\label{discr-crit}
Let $\calE$, $\calT$ and~$\Si$ be as in Proposition~\ref{structure-E}.
Let $\bar{\Si}\subset\La=\<j_1,j_2\>$ be the image of~$\Si$ under the projection~$\calE\to\La$.
Let $T_1,T_2,H$ be the Heisenberg translations by~$(v_1,t_1)$, $(v_2,t_2)$, $(0,\nu)$
and $w_1=\frac{1}{2}(v_2-v_1)$,
$w_2=\frac{1}{2}(v_1+v_2)$
as in Proposition~\ref{structure-E}.
We will reformulate the conditions of Lemma~\ref{f(0)}:
\begin{enumerate}[(1)]
\item
The condition $|f(0)|\ge2$ for~$f=j_2$ is equivalent to $r\sin(\pi/n_2)\ge1$.
\item
The conditions $r\sin(\pi/n_2)\ge1$
and
\[
  g(u,v)
  =(u-a)^2+d\cdot(v-b)^2-\frac{4r^2\sin^2(\pi/n_2)}{|w_1|^2}
  \ge0
\]
for every~$\bar{\si}\in\bar{\Si}$ and~$(u,v)\in\z^2$ with $u\equiv v\mod2$
satisfying
\[
  |u-a|<\frac{2r\sin(\pi/n_2)}{|w_1|},\quad
  |v-b|<\frac{2r\sin(\pi/n_2)}{|w_1|\sqrt{d}},
\]
where 
\begin{align*}
  p&=\bar{\si}(0),\quad
  a=-\frac{\Re(p\bar{w}_1)}{|w_1|^2},\quad
  b=-\frac{\Re(p\bar{w}_2)}{|w_2|^2}
%=\frac{\pm\Im(p\bar{w}_1}{|w_1|^2},
\end{align*}
and
\begin{align*}
  d&=3~\text{for}~(n_1,n_2)=(3,3),(2,3),(2,6),(3,6),\\ %(3,3,3) and (2,3,6)-group
  d&=1~\text{for}~(n_1,n_2)=(2,4),(4,4). %(2,4,4)-group
\end{align*}
imply
$|f(0)|\ge2$ for all~$f\in\La\backslash\{\Id\}$.
\item
The condition
$\abs{h(0)}\ge2$ for all vertical Heisenberg translations~$h\in\calE\backslash\{\Id\}$
is equivalent to~$\nu\ge2$.
\end{enumerate}
\end{proposition}

\begin{proof}
%We can write every element~$f$ in~$\La$ as a word in the generators~$j_1$ and~$j_2$.
The group $\<j_1,j_2\>$ is the projection to $\c$ of the group $\calE$.
Projecting~$\iota_1$ and~$\iota_2$ to~$\c$,
we obtain the rotations~$j_1$ and~$j_2$ of~$\c$
through~$2\pi/n_1$ and~$2\pi/n_2$ around $\varphi_1$ and $\varphi_2$ respectively.
These rotations are given by
$j_k(z)=\mu_k\cdot z+(1-\mu_k)\cdot\varphi_k$ for~$k=1,2$,
where $\mu_k=\exp(2\pi i/n_k)$.
\begin{enumerate}[(1)]
\item
We have $j_2(z)=\mu_2\cdot z+(1-\mu_2)\cdot\varphi_2$, therefore
$j_2(0)=(1-\mu_2)\cdot\varphi_2$
and
\[
  |j_2(0)|
  =|1-\mu_2|\cdot|\varphi_2|
  =|1-\exp(2\pi i/n_2)|\cdot r
  =2r\sin(\pi/n_2).
\]
\item
According to Proposition~\ref{structure-E},
every element of~$\calE$ is of the form $T_1^xT_2^yH^n\si$ for some $x,y,n\in\z$ and $\si\in\Si$.
The projection to~$\c$ maps $H$ to the identity,
$T_k$ to the Euclidean translation by $v_k$,
$T_1^xT_2^yH^n$ to the Euclidean translation by $x v_1+y v_2$ and $\si$ to a rotation~$\bar{\si}\in\bar{\Si}$ respectively.
Therefore every element of $\La$ is a product of a translation by $xv_1+yv_2$ for some~$x,y\in\z$ and a rotation in~$\bar{\Si}$.
Hence every point in the orbit of~$0$ under $\La$ is of the form $p+xv_1+yv_2$, where $x,y\in\z$ and
$p=\bar{\si}(0)$ for some $\bar{\si}\in\bar{\Si}$.
The vectors $w_1=\frac{1}{2}(v_2-v_1)$
and $w_2=\frac{1}{2}(v_1+v_2)$ introduced in Proposition~\ref{structure-E} will be more useful in calculations than~$v_1$ and~$v_2$.
We know that
\[w_2=\pm i w_1\sqrt{d},\]
where
\begin{align*}
  d&=3~\text{for}~(n_1,n_2)=(3,3),(2,3),(2,6),(3,6),\\ %(2,3,6)-group
  d&=1~\text{for}~(n_1,n_2)=(2,4),(4,4). %(2,4,4)-group.
\end{align*}
We can rewrite every point in the orbit of~$0$ under $\La$ as
\begin{align*}
  p+xv_1+yv_2
  &=p+x(-w_1+w_2)+y(w_1+w_2)\\
  &=p+(y-x)w_1+(x+y)w_2\\
  &=p+uw_1+vw_2,
\end{align*}
where $u=y-x$ and $v=x+y$.
Points $(x,y)\in\z^2$ are mapped to points $(u,v)\in\z^2$ with $u\equiv v\mod 2$.
The condition $|f(0)|\ge2$ is equivalent to
\begin{align*}
&|p+uw_1+vw_2|^2\\
&=u^2|w_1|^2+v^2|w_2|^2+2uv\Re(w_1\bar{w}_2)+2u\Re(p\bar{w}_1)+2v\Re(p\bar{w}_2)+|p|^2.
\end{align*}
Note that $w_2=\pm i w_1\sqrt{d}$ implies $\Re(w_1\bar{w}_2)=0$, hence
\begin{align*}
&|p+uw_1+vw_2|^2\\
&=u^2|w_1|^2+v^2|w_2|^2+2u\Re(p\bar{w}_1)+2v\Re(p\bar{w}_2)+|p|^2.
\end{align*}
Setting
\begin{align*}
  &a=-\frac{\Re(p\bar{w}_1)}{|w_1|^2},\quad
  b=-\frac{\Re(p\bar{w}_2)}{|w_2|^2},
\end{align*}
we can rewrite this as 
\begin{align*}
&|p+uw_1+vw_2|^2\\
&=u^2|w_1|^2+v^2|w_2|^2-2ua|w_1|^2-2vb|w_2|^2+|p|^2\\
&=|w_1|^2\cdot(u^2-2ua)+|w_2|^2\cdot(v^2-2vb)+|p|^2\\
&=|w_1|^2\cdot(u^2-2ua+a^2)+|w_2|^2\cdot(v^2-2vb+b^2)+(|p|^2-a^2|w_1|^2-b^2|w_2|^2)\\
&=|w_1|^2\cdot\bigg((u-a)^2+d\cdot(v-b)^2\bigg)+\bigg(|p|^2-(a^2+db^2)|w_1|^2\bigg).
\end{align*}
We note that $w_2=\pm i w_1\sqrt{d}$ implies
\[\Re(p\bar{w}_2)=\pm\sqrt{d}\Im(p\bar{w}_1),\]
%\begin{align*}
%  \Re(p\bar{w}_2)
%  &=\Re(p(\mp i \bar{w_1}\sqrt{d}))
%  =\mp\sqrt{d}\Re(i p\bar{w_1})\\
%  &=\mp\sqrt{d}(-\Im(p\bar{w_1}))
%  =\pm\sqrt{d}\Im(p\bar{w}_1),
%\end{align*}
hence
\begin{align*}
    a^2+db^2
    &=\frac{\Re^2(p\bar{w}_1)}{|w_1|^4}
    +d\cdot\frac{\Re^2(p\bar{w}_2)}{|w_2|^4}
    =\frac{\Re^2(p\bar{w}_1)}{|w_1|^4}
    +d\cdot\frac{d\cdot\Im^2(p\bar{w}_1)}{d^2\cdot|w_1|^4}\\
    &=\frac{\Re^2(p\bar{w}_1)+\Im^2(p\bar{w}_1)}{|w_1|^4}
    =\frac{|p\bar{w}_1|^2}{|w_1|^4}
    =\frac{|p|^2}{|w_1|^2}.
\end{align*}
Therefore
\begin{align*}
|p+uw_1+vw_2|^2
=|w_1|^2\cdot((u-a)^2+d\cdot(v-b)^2).
\end{align*}
We will test the condition
\[|p+uw_1+vw_2|\ge2r\sin(\pi/n_2)\]
for all~$(u,v)\in\z^2$ with~$u\equiv v\mod2$
(excluding the case $p=0$, $u=v=0$)
%that corresponds to $f=\Id$
which, combined with $r\sin(\pi/n_2)\ge1$, 
would imply
\[|p+uw_1+vw_2|\ge2.\]
The condition
\[|p+uw_1+vw_2|\ge2r\sin(\pi/n_2)\]
is equivalent to
\[
  (u-a)^2+d\cdot(v-b)^2
  \ge\frac{4r^2\sin^2(\pi/n_2)}{|w_1|^2}
\]
for all~$(u,v)\in\z^2$ with~$u\equiv v\mod2$,
excluding the case $a=b=u=v=0$.
Note that this inequality follows immediately
if
\[
  |u-a|\ge\frac{2r\sin(\pi/n_2)}{|w_1|}
  \quad\text{or}\quad
  |v-b|\ge\frac{2r\sin(\pi/n_2)}{|w_1|\sqrt{d}},
\]
so we only need to check that
\[
  g(u,v)
  =(u-a)^2+d(v-b)^2-\frac{4r^2\sin^2(\pi/n_2)}{|w_1|^2}\ge0
\]
for all $(u,v)\in\z^2$ with $u\equiv v\mod2$ inside the rectangle where
\[
  |u-a|<\frac{2r\sin(\pi/n_2)}{|w_1|},\quad
  |v-b|<\frac{2r\sin(\pi/n_2)}{|w_1|\sqrt{d}}.
\]
\item
The second condition is $|h(0)|\ge2$
for every vertical translation~$h$ in $\<\iota_1,\iota_2\>$.
We know that~$h=H^n$, where $H$ is the vertical translation by~$(0,\nu)$.
%by $(0,96r^2\sqrt{3}\cos^2(\theta))$.
We need $n\cdot\nu\ge2$ for all $n\not=0$,
i.e.\ $\nu\ge2$.\qedhere
\end{enumerate}
\end{proof}
\noindent
For the purposes of the following calculations, we will denote $t=\tan(\theta)$ and $\tau=\tan(\pi/12)$.

\subsection{Proof of Proposition \ref{prop-discr}}

\begin{proof}
We will show that under the assumptions on~$\al$ in Proposition~\ref{prop-discr} the condition 
\[r=\cosh(m/2)\ge\frac{1}{\sin(\pi/n_2)}\]
is sufficient for $\Ga$ to be discrete.
To this end, we will check the conditions of Proposition~\ref{discr-crit}.
\\
\linebreak
We will show that the second condition of Proposition~\ref{discr-crit}, $\nu\ge2$, is always satisfied:
According to Proposition~\ref{structure-E}(a),
\[\nu\ge2R^2\sqrt{3}=8r^2\cos^2(\theta)\sqrt{3}.\]
We have $n_2\in\{3,4,6\}$ and hence
\[
  r^2
  \ge\frac{1}{\sin^2(\pi/n_2)}
  \ge\frac{1}{\sin^2(\pi/3)}
  =\frac43.
\]
In all the cases we assume that $|t|\le2\sqrt{2}-\sqrt{3}$ and hence
\[
  \cos^2(\theta)=\frac{1}{t^2+1}
  \ge\frac{\sqrt{3}+\sqrt{2}}{4\sqrt{3}}. %\approx0{.}454
\]
Therefore
\begin{align*}
  |\nu|
  \ge8r^2\cos^2(\theta)\sqrt{3}
%  \ge8\cdot\frac43\cdot\frac{\sqrt{3}+\sqrt{2}}{4\sqrt{3}}\cdot\sqrt{3}
  \ge\frac{8}{3}\cdot(\sqrt{3}+\sqrt{2})>2.
\end{align*}
\iffalse
$(n_1,n_2)=(2,3)$:
$r\ge\frac{2}{\sqrt{3}}, \cos(\theta)\ge\frac{\sqrt{3}}{2}
\Longrightarrow
\nu
=96r^2\cos^2(\theta)\sqrt{3}
\ge96\cdot\frac{4}{3}\cdot\frac{3}{4}\cdot\sqrt{3}=96\sqrt{3}>2
$.

$(n_1,n_2)=(2,4)$:
$r\ge\sqrt{2},
\cos(\theta)\ge\frac{1+\sqrt{3}}{2\sqrt{2}}
\Longrightarrow
\nu=64r^2\cos^2(\theta)
\ge64\cdot2\cdot\frac{2+\sqrt{3}}{4}>2
$.

$(n_1,n_2)=(4,4)$:
$
r\ge\sqrt{2},
\cos(\theta)\ge\frac{1+\sqrt{3}}{2\sqrt{2}}
\Longrightarrow
\nu=16r^2\cos^2(\theta)
\ge16\cdot2\cdot\frac{2+\sqrt{3}}{4}>2
$.

$(n_1,n_2)=(2,6)$:
$
r\ge2,
\cos(\theta)\ge\frac{1+\sqrt{3}}{2\sqrt{2}}
\Longrightarrow
\nu=16r^2\cos^2(\theta)\sqrt{3}
\ge16\cdot4\cdot\frac{2+\sqrt{3}}{4}\sqrt{3}>2
$.

$\{n_1,n_2\}=\{3,6\}$:
$
r\ge2,
\cos(\theta)\ge\frac{1+\sqrt{3}}{2\sqrt{2}}
\Longrightarrow
\nu=8r^2\cos^2(\theta)\sqrt{3}
\ge8\cdot4\cdot\frac{2+\sqrt{3}}{4}\sqrt{3}>2
$.
\fi
\noindent
It remains to check that the first condition of Proposition~\ref{discr-crit},
\[
  g(u,v)
  =(u-a)^2+d\cdot(v-b)^2-\frac{4r^2\sin^2(\pi/n_2)}{|w_1|^2}
  \ge0
\]
for every~$\bar{\si}\in\bar{\Si}$ and~$(u,v)\in\z^2$ with $u\equiv v\mod2$ and
\begin{align*}
  &u\in\left(a-\frac{2r\sin(\pi/n_2)}{|w_1|},a+\frac{2r\sin(\pi/n_2)}{|w_1|}\right),\\
  &v\in\left(b-\frac{2r\sin(\pi/n_2)}{|w_1|\sqrt{d}},b+\frac{2r\sin(\pi/n_2)}{|w_1|\sqrt{d}}\right),
\end{align*}
is satisfied in each of the six cases.
For each pair~$(n_1,n_2)$ and for each $p=\bar{\si}(0)$, we will calculate~$a$ and~$b$, list the bounds on~$a$ and~$b$ and hence determine the bounds on~$u$ and~$v$.
We then check that $g(u,v)\ge0$ for all~$(u,v)$ satisfying the conditions above.
\begin{enumerate}[(1)]
\item
Case $(n_1,n_2)=(3,3)$:
This case has already been discussed in~\cite{povprat}.
We include it here for completeness and to show that the condition $|t|\le1/\sqrt{3}$ in this proof cannot be relaxed.
We list the values of~$a$, $b$ and~$a^2+3b^2$ for all $p=\bar{\si}(0)$ with $\si\in\Si$:
\begin{center}
\def\arraystretch{1.5}
\begin{tabular}{|c|c|c|c| }
\hline
\(p\) & \(a\) & \(b\) & \(a^2+3b^2\) \\
\hline
\(0\) & \(0\) & \(0\) & \(0\) \\
\hline
$j_1(0)$ & $-1$ & $-t/\sqrt{3}$ & $t^2+1$ \\
\hline
$j_{11}(0)$ & $\frac{1}{2}(t\sqrt{3}-1)$ & $-\frac16(t\sqrt{3}+3)$ & $t^2+1$ \\
\hline
\end{tabular}
\end{center}
We calculate
\[g(u,v)=u^2+3v^2-2au-6bv+(a^2+3b^2)-(t^2+1).\]
We only need to consider
$u\in(a-\ga\sqrt{3},a+\ga\sqrt{3})$,
$v\in(b-\ga,b+\ga)$,
where $\ga=\sqrt{\frac{t^2+1}{3}}$.
We assume $|t|\le1/\sqrt{3}$.
\begin{itemize}
\item
$p=0$:
We consider $u,v\in(-1,1)$.
The only such~$u,v\in\z$ with~$u\equiv v\mod2$
are $u=v=0$,
which corresponds to the excluded case $f=\Id$.
\item
$p=j_1(0)$:
We consider $u\in(-3,1)$, $v\in[-1,1]$.
The only such $u,v\in\z$ with~$u\equiv v\mod2$
are $(u,v)=(-2,0),(-1,-1),(-1,1),(0,0)$.
The function
\[g(u,v)=u^2+3v^2+2u+2tv\sqrt{3}\]
is non-negative:
For $|t|\le1/\sqrt{3}$,
\begin{align*}
  &g(-2,0)=0,\quad
  g(-1,-1)=-2\sqrt{3}\left(t-\frac{1}{\sqrt{3}}\right)\ge0,\\
  &g(0,0)=0,\quad
  g(-1,1)=2\sqrt{3}\left(t+\frac{1}{\sqrt{3}}\right)\ge0.
\end{align*}
Note that this discreteness proof does not work for $|t|>1/\sqrt{3}$ as then we have either $g(-1,-1)<0$ or $g(-1,1)<0$.
\item
$p=j_{11}(0)$:
We consider $u\in(-3,2)$, $v\in(-2,1)$.
The only such~$u,v\in\z$ with~$u\equiv v\mod2$
are $(u,v)=(-2,0),(-1,-1),(0,0),(1,-1)$.
The function
\[g(u,v)=u^2+3v^2-(t\sqrt{3}-1)u+(t\sqrt{3}+3)v\]
is non-negative:
For $|t|\le1/\sqrt{3}$,
\begin{align*}
  &g(-1,-1)=0,\quad
  g(-2,0)=2\sqrt{3}\left(t+\frac{1}{\sqrt{3}}\right)\ge0,\\
  &g(0,0)=0,\quad
  g(1,-1)=-2\sqrt{3}\left(t-\frac{1}{\sqrt{3}}\right)\ge0.
\end{align*}
Note that this discreteness proof does not work for $|t|>1/\sqrt{3}$ as then we have either $g(-2,0)<0$ or $g(1,-1)<0$.
\end{itemize}
\item
Case $(n_1,n_2)=(2,3)$:
%Figure \ref{fig2} shows the points $f(0)$ for all words $f$ of length up to $6$ in the case $r=1$ and $\theta=0$.
%\begin{figure}[h]
%\centering
%\includegraphics[width=0.5\textwidth]{order232.png}
%\caption{Points $f(0)$ for all words $f$ up to length $6$.}
%\label{fig2}
%\end{figure}
We list the values of~$a$, $b$ and~$a^2+3b^2$ for all $p=\bar{\si}(0)$ with $\si\in\Si$: 
\begin{center}
\def\arraystretch{1.5}
\begin{tabular}{|c|c|c|c|c| }
\hline\(p\) & \(a\) & \(b\) & \(a^2+3b^2\) \\
\hline\(0\) & \(0\) & \(0\) & \(0\) \\
\hline\(j_1(0)\) & \(\frac{1}{2\sqrt{3}}\left(t-\sqrt{3}\right)\) & \(-\frac{1}{2\sqrt{3}}\left(t+\frac{1}{\sqrt{3}}\right)\) & \(\frac{1}{3}\left(t^2+1\right)\) \\
\hline\(j_2(0)\) & \(\frac{1}{2}\) & \(-\frac{t}{2\sqrt{3}}\) & \(\frac{1}{4}\left(t^2+1\right)\) \\
\hline\(j_{12}(0)\) & \(\frac{1}{2\sqrt{3}}\left(t-2\sqrt{3}\right)\) & \(-\frac{1}{6}\) & \(\frac{1}{12}\left(t^2-4t\sqrt{3}+13\right)\) \\
\hline\(j_{22}(0)\) & \(\frac{1}{4}(1+t\sqrt{3})\) & \(\frac{1}{4\sqrt{3}}\left(\sqrt{3}-t\right)\) & \(\frac{1}{4}\left(t^2+1\right)\) \\
\hline\(j_{122}(0)\) & \(-\frac{1}{4\sqrt{3}}\left(t+3\sqrt{3}\right)\) &\(-\frac{1}{4\sqrt{3}}\left(t+\frac{5}{\sqrt{3}}\right)\) & \(\frac{1}{12}\left(t^2+4t\sqrt{3}+13\right)\) \\
\hline
\end{tabular}
\end{center}
We calculate
\[g(u,v)=u^2+3v^2-2au-6bv+(a^2+3b^2)-\frac14(t^2+1).\]
We only need to consider
\[
  u\in(a-\ga\sqrt{3},a+\ga\sqrt{3}),\quad
  v\in(b-\ga,b+\ga),\quad\text{where}~
  \ga=\frac{\sqrt{t^2+1}}{2\sqrt{3}}.
\]
We assume $|t|\le2\sqrt{2}-\sqrt{3}$.
\begin{itemize}
\item
$p=0$:
We consider $u,v\in(-1,1)$.
The only such~$u,v\in\z$ are $u=v=0$,
which corresponds to the excluded case $f=\Id$.
\item
$p=j_1(0)$:
We consider $u\in(-2,1)$, $v\in(-1,1)$.
The only such~$u,v\in\z$ with~$u\equiv v\mod2$
are $u=v=0$.
The function
\[
  g(u,v)=u^2+3v^2+\frac{u}{\sqrt{3}}\left(\sqrt{3}-t\right)+v\left(1+t\sqrt{3}\right)+\frac{1}{12}\left(t^2+1\right)
\]
is non-negative:
\[g(0,0)=\frac{1}{12}\left(t^2+1\right)>0.\]
\item
$p=j_2(0)$:
We consider $u\in(-1,2)$, $v\in(-1,1)$.
The only such~$u,v\in\z$ with~$u\equiv v\mod2$
are $u=v=0$.
The function
\[g(u,v)=u^2+3v^2-u+tv\sqrt{3}\]
is non-negative: $g(0,0)=0$.
\item
$p=j_{12}(0)$:
We consider $u\in(-3,1)$, $v\in(-1,1)$.
The only such~$u,v\in\z$ with~$u\equiv v\mod2$ are $(u,v)=(-2,0),(0,0)$.
The function
\[
  g(u,v)=u^2+3v^2-\frac{u}{\sqrt{3}}\left(t-2\sqrt{3}\right)+v
  -\frac16\left(t^2+2t\sqrt{3}-5\right)
\]
is non-negative:
For $|t|\le2\sqrt{2}-\sqrt{3}$,
\begin{align*}
  g(0,0)
  %&=-\frac16(t^2+2t\sqrt{3}-5)
  &=-\frac16\left(t-(2\sqrt{2}-\sqrt{3})\right)\left(t-(2\sqrt{2}+\sqrt{3})\right)\ge0,\\
  g(-2,0)
  %&=-\frac16(t^2-2t\sqrt{3}-5)
  &=-\frac16\left(t+(2\sqrt{2}-\sqrt{3})\right)\left(t-(2\sqrt{2}+\sqrt{3})\right)\ge0.
\end{align*}
Note that this discreteness proof does not work for $|t|>2\sqrt{2}-\sqrt{3}$ as then we have either $g(0,0)<0$ or $g(-2,0)<0$.
\item
$p=j_{22}(0)$:
We consider $u\in(-1,2)$, $v\in(-1,1)$.
The only such~$u,v\in\z$ with~$u\equiv v\mod2$
are $u=v=0$.
The function
\[
  g(u,v)
  =u^2+3v^2
  -\frac{u}{2}\left(1+t\sqrt{3}\right)
  -\frac{v}{2}\left(3-t\sqrt{3}\right)
\]
is non-negative: $g(0,0)=0$.
\item
$p=j_{122}(0)$:
We consider $u,v\in(-2,1)$.
The only such~$u,v\in\z$ with~$u\equiv v\mod2$ are $(u,v)=(-1,-1),(0,0)$.
The function
\[
  g(u,v)
  =u^2+3v^2
  +\frac{u}{2\sqrt{3}}\left(t+3\sqrt{3}\right)
  +\frac{v}{2}\left(t\sqrt{3}+5\right)
  -\frac16\left(t^2-2t\sqrt{3}-5\right)
\]
is non-negative:
For $|t|\le2\sqrt{2}-\sqrt{3}$,
\begin{align*}
  g(-1,-1)
  %&=-\frac16(t^2+2t\sqrt{3}-5)
  &=\left(t-(2\sqrt{2}-\sqrt{3})\right)\left(t-(2\sqrt{2}+\sqrt{3})\right)\ge0,\\
  g(0,0)
  %&=-\frac16(t^2-2t\sqrt{3}-5)\ge0.
  &=\left(t+(2\sqrt{2}-\sqrt{3})\right)\left(t-(2\sqrt{2}+\sqrt{3})\right)\ge0.
\end{align*}
Note that this discreteness proof does not work for $|t|>2\sqrt{2}-\sqrt{3}$ as then we have either $g(-1,-1)<0$ or $g(0,0)<0$.
\end{itemize}
\item
Case $(n_1,n_2)=(2,6)$:
We list the values of~$a$, $b$ and~$a^2+3b^2$ for all $p=\bar{\si}(0)$ with $\si\in\Si$:
\begin{center}
\def\arraystretch{1.5}
\begin{tabular}{|c|c|c|c|c| }
\hline\(p\) & \(a\) & \(b\) & \(a^2+3b^2\) \\
\hline\(0\) & \(0\) & \(0\) & \(0\) \\
\hline\(j_1(0)\) & \(1\) & \(-\frac{t}{\sqrt{3}}\) & \(t^2+1\) \\
\hline\(j_2(0)\) & \(\frac{1}{4}\left(t\sqrt{3}-1\right)\) & \(-\frac{1}{4\sqrt{3}}\left(t+\sqrt{3}\right)\) & \(\frac{1}{4}\left(t^2+1\right)\) \\
\hline\(j_{21}(0)\) & \(\frac{1}{4}\left(1-t\sqrt{3}\right)\) & \(-\frac{1}{4}\left(t\sqrt{3}+3\right)\) & \(\frac{1}{4}\left(3t^2+4t\sqrt{3}+7\right)\) \\
\hline\(j_{22}(0)\) & \(\frac{1}{4}(t\sqrt{3}-3)\) & \(-\frac{1}{4}\left(t\sqrt{3}+1\right)\) & \(\frac{3}{4}\left(t^2+1\right)\) \\
\hline\(j_{221}(0)\) & \(-\frac{1}{4}\left(t\sqrt{3}+5\right)\) &\(-\frac{1}{4\sqrt{3}}\left(t+3\sqrt{3}\right)\) & \(\frac{1}{4}\left(t^2+4t\sqrt{3}+13\right)\) \\
\hline
\end{tabular}
\end{center}
We calculate
\[g(u,v)=u^2+3v^2-2au-6bv+(a^2+3b^2)-\frac14(t^2+1).\]
We only need to consider
\[
  u\in(a-\ga\sqrt{3},a+\ga\sqrt{3}),\quad
  v\in(b-\ga,b+\ga),~\text{where}~
  \ga=\frac{\sqrt{t^2+1}}{2\sqrt{3}}.
\]
We assume $|t|\le(4-\sqrt{5})/\sqrt{3}$.
\begin{itemize}
\item
$p=0$:
We consider $u,v\in(-1,1)$.
The only such $u,v\in\z$ are $(u,v)=(0,0)$,
which corresponds to the excluded case $f=\Id$.
\item
$p=j_1(0)$:
We consider $u\in(0,2)$, $v\in(-1,1)$.
There are no such $u,v\in\z$ with $u\equiv v\mod2$.

Note that this discreteness proof does not work for $|t|>(4-\sqrt{5})/\sqrt{3}$:
We have
\[
  g(u,v)
  =u^2+3v^2-2u+2t\sqrt{3}v+\frac34(t^2+1),
\]
and
\begin{align*}
  g(1,\pm1)
%  &=\frac14\left(3t^2\pm8t\sqrt{3}+11\right)\\
  &=\frac34\left(t\pm\frac{4-\sqrt{5}}{\sqrt{3}}\right)
  \left(t\pm\frac{4+\sqrt{5}}{\sqrt{3}}\right),
\end{align*}
hence for $|t|>(4-\sqrt{5})/\sqrt{3}$
we have either $g(1,1)<0$ or $g(1,-1)<0$.
\item
$p=j_2(0)$:
We consider
$u\in(-2,2)$, $v\in(-1,1)$.
The only such $u,v\in\z$ are $(u,v)=(0,0)$.
The function
\[g(u,v)=u^2+3v^2+\frac12(1-t\sqrt{3})u+\frac12(3-t\sqrt{3})v\]
is non-negative: $g(0,0)=0$.
\item
$p=j_{21}(0)$:
We consider $u\in(-2,2)$, $v\in(-3,1)$.
The only such $u,v\in\z$ with $u\equiv v\mod2$ are
$(u,v)=(-1,-1),(1,-1),(0,0),(0,-2)$.
The function
\[
  g(u,v)=u^2+3v^2+\frac12(t\sqrt{3}-1)u+\frac32(t\sqrt{3}+3)v+\frac12(t+\sqrt{3})^2
\]
is non-negative:
For $|t|\le(4-\sqrt{5})/\sqrt{3}$, we have
\begin{align*}
  &g(-1,-1)=\frac12(t-\sqrt{3})^2>0,\quad
  g(1,-1)=\frac12(t^2+1)>0,\\
  &g(0,0)=\frac12(t+\sqrt{3})^2>0,\quad
  g(0,-2)=\frac12(t-\sqrt{3})(t-3\sqrt{3})>0.
\end{align*}
\item
$p=j_{22}(0)$:
We consider
$u\in(-3,1)$, $v\in(-2,2)$.
The only such $u,v\in\z$ with $u\equiv v\mod2$ are
$(u,v)=(-1,-1),(-1,1),(0,0),(-2,0)$.
The function
\[
  g(u,v)
  =u^2+3v^2
  -\frac12(t\sqrt{3}-3)u+\frac{3}{2}(t\sqrt{3}+1)v
  +\frac12(t^2+1)
\]
is non-negative:
For $|t|\le(4-\sqrt{5})/\sqrt{3}$, we have
\begin{align*}
  &g(-1,-1)=\frac12(t-\sqrt{3})^2>0,\quad
  g(-1,1)=\frac12(t+\sqrt{3})(t+3\sqrt{3})>0,\\
  &g(0,0)=\frac12(t^2+1)>0,\quad
  g(0,-2)=\frac12(t+\sqrt{3})^2>0.
\end{align*}
\item
$p=j_{221}(0)$:
We consider $u\in(-3,1)$, $v\in(-2,0)$.
The only such $u,v\in\z$ with $u\equiv v\mod2$
are $(u,v)=(-1,-1)$.
The function
\[
  g(u,v)
  =u^2+3v^2
  -\frac12\left(5+t\sqrt{3}\right)u
  +\frac12\left(9+t\sqrt{3}\right)v
  +\left(3+t\sqrt{3}\right)
\]
is non-negative: $g(-1,-1)=0$.
\end{itemize}
\item 
Case $(n_1,n_2)=(3,6)$:
We list the values of~$a$, $b$ and $a^2+3b^2$ for all $p=\bar{\si}(0)$ with $\si\in\Si$:
\begin{center}
\def\arraystretch{1.5}
\begin{tabular}{|c|c|c|c|c| }
\hline\(p\) & \(a\) & \(b\) & \(a^2+3b^2\) \\
\hline\(0\) & \(0\) & \(0\) & \(0\) \\
\hline\(j_2(0)\) & \(-\frac{1}{2\sqrt{3}}\left(t+\sqrt{3}\right)\) & \(\frac{1}{2\sqrt{3}}\left(t-\frac{1}{\sqrt{3}}\right)\) & \(\frac{1}{3}\left(t^2+1\right)\) \\
\hline\(j_{22}(0)\) & \(-\frac{1}{2}\left(t\sqrt{3}+1\right)\) & \(\frac{1}{2\sqrt{3}}\left(t-\sqrt{3}\right)\) & \(t^2+1\) \\
\hline\(j_{112}(0)\) & \(-\frac{1}{2\sqrt{3}}\left(t+\sqrt{3}\right)\) & \(\frac{1}{2\sqrt{3}}\left(\frac{5}{\sqrt{3}}-t\right)\) & \(\frac{1}{3}\left(t^2-2t\sqrt{3}+7\right)\) \\
\hline\(j_{221}(0)\) & \(-\frac{1}{2}\left(t\sqrt{3}+3\right)\) & \(-\frac{1}{2\sqrt{3}}(t+\sqrt{3})\) & \(t^2+2t\sqrt{3}+3\) \\
\hline\(j_{222}(0)\) & \(-\frac{2t}{\sqrt{3}}\) & \(-\frac{2}{3}\) & \(\frac{4}{3}\left(t^2+1\right)\) \\
\hline
\end{tabular}
\end{center}
We calculate
\[g(u,v)=u^2+3v^2-2au-6bv+(a^2+3b^2)-\frac13(t^2+1).\]
We only need to consider
\[
  u\in(a-\ga\sqrt{3},a+\ga\sqrt{3}),\quad
  v\in(b-\ga,b+\ga),\quad\text{where}~
  \ga=\frac{\sqrt{t^2+1}}{3}.
\]
We assume $|t|\le1/\sqrt{3}$.
\begin{itemize}
\item
$p=0$:
We consider $u,v\in(-1,1)$.
The only such $u,v\in\z$ are $u=v=0$,
which corresponds to the excluded case $f=\Id$.
\item
$p=j_2(0)$:
We consider $u\in(-2,1)$, $v\in(-1,1)$.
The only such $u,v\in\z$ with $u\equiv v\mod2$ are $u=v=0$.
The function
\[
  g(u,v)
  =u^2+3v^2
  +\left(1+\frac{t}{\sqrt{3}}\right)u+(1-t\sqrt{3})v
\]
is non-negative: $g(0,0)=0$. 
\item
$p=j_{22}(0)$:
We consider $u,v\in(-2,1)$.
The only such $u,v\in\z$ with $u\equiv v\mod2$ are $(u,v)=(0,0),(-1,-1)$.
The function
\[
  g(u,v)
  =u^2+3v^2
  +\left(t\sqrt{3}+1\right)u-(t\sqrt{3}-3)v
  +\frac23(t^2+1)
\]
is non-negative: 
\[g(0,0)=g(-1,-1)=\frac23(t^2+1)>0.\]
\item
$p=j_{112}(0)$:
We consider $u\in(-2,1)$, $v\in(0,2)$.
The only such $u,v\in\z$ with $u\equiv v\mod2$ is $(u,v)=(-1,1)$.
The function
\[
  g(u,v)=u^2+3v^2+\left(1+\frac{t}{\sqrt{3}}\right)u-(5-t\sqrt{3})v+\left(2-\frac{2t}{\sqrt{3}}\right)
\]
is non-negative: $g(-1,1)=0$. 
\item
$p=j_{221}(0)$:
We consider $u\in(-3,0)$, $v\in(-2,1)$.
The only such $u,v\in\z$ with $u\equiv v\mod2$ are $(u,v)=(-2,0),(-1,-1)$.
The function
\[
  g(u,v)
  =u^2+3v^2
  +\left(t\sqrt{3}+3\right)(u+v)
  +\frac23(t^2+3t\sqrt{3}+4)
\]
is non-negative: 
\[g(-2,0)=g(-1,-1)=\frac23(t^2+1)>0.\]
\item
$p=j_{222}(0)$:
We consider
$u\in(-2,2)$, $v\in(-2,0)$.
The only such $u,v\in\z$ with $u\equiv v\mod2$ are $(u,v)=(-1,-1),(1,-1)$.
The function
\[
  g(u,v)
  =u^2+3v^2+\frac{4t}{\sqrt{3}}u+4v+(t^2+1)
\]
is non-negative:
For $|t|\le1/\sqrt{3}$,
\begin{align*}
 g(\pm1,-1)
 %&=t^2\pm\frac{4t}{\sqrt{3}}+1
 &=\left(t\pm\sqrt{3}\right)\left(t\pm\frac{1}{\sqrt{3}}\right)\ge0.
\end{align*}

Note that this discreteness proof does not work for $|t|>1/\sqrt{3}$
as then we have either $g(1,-1)<0$ or $g(-1,-1)<0$.
\end{itemize}
\item
Case $(n_1,n_2)=(2,4)$:
We list the values of~$a$, $b$ and~$a^2+b^2$ for all $p=\bar{\si}(0)$ with $\si\in\Si$:
\begin{center}
\def\arraystretch{1.5}
\begin{tabular}{|c|c|c|c|c| }
\hline\(p\) & \(a\) & \(b\) & \(a^2+b^2\) \\
\hline\(0\) & \(0\) & \(0\) & \(0\) \\
\hline\(j_1(0)\) & \(\frac{1}{2}\left(t-1\right)\) & \(-\frac{1}{2}\left(t+1\right)\) & \(\frac{1}{2}\left(t^2+1\right)\) \\
\hline\(j_2(0)\) & \(\frac{1}{2}\) & \(-\frac{t}{2}\) & \(\frac{1}{4}\left(t^2+1\right)\) \\
\hline\(j_{21}(0)\) & \(\frac{1}{2}\left(t+2\right)\) & \(-\frac{1}{2}\) & \(\frac{1}{4}\left(t^2+4t+5\right)\) \\
\hline
\end{tabular}
\end{center}
We calculate
\[g(u,v)=u^2+v^2-2au-2bv+(a^2+b^2)-\frac14(t^2+1).\]
We only need to consider
$u\in(a-\ga,a+\ga)$,
$v\in(b-\ga,b+\ga)$,
where $\ga=\frac12\sqrt{t^2+1}$.
%The assumption $|\al-\pi|\le\pi/4$ implies $|t|\le1$.
We assume~$|t|\le1$.
\begin{itemize}
\item
$p=0$:
We consider $u,v\in(-1,1)$.
The only such~$u,v\in\z$ with~$u\equiv v\mod2$
are $u=v=0$,
which corresponds to the excluded case $f=\Id$.
\item
$p=j_1(0)$:
We consider $u,v\in(-2,1)$.
The only such~$u,v\in\z$ with~$u\equiv v\mod2$
are $u=v=-1$ and $u=v=0$.
The function
\[g(u,v)=u^2+v^2+\left(1-t\right)u+\left(1+t\right)v+\frac14(t^2+1)\]
is non-negative:
\[g(-1,-1)=g(0,0)=\frac14(t^2+1)>0.\]
\item
$p=j_2(0)$:
We consider $u\in(-1,2)$, $v\in(-2,2)$.
The only such~$u,v\in\z$ with~$u\equiv v\mod2$
are $(u,v)=(0,0),(1,1),(1,-1)$.
The function
\[g(u,v)=u^2+v^2-u+tv\]
is non-negative:
$g(0,0)=0$, $g(1,\pm 1)=1\pm t\ge0$.

Note that this discreteness proof does not work for $|t|>1$ as then we have either $g(1,1)<0$ or $g(1,-1)<0$.
\item
$p=j_{21}(0)$,
We consider $u\in(-1,3)$, $v\in(-2,1)$.
The only such~$u,v\in\z$ with~$u\equiv v\mod2$
are $(u,v)=(0,0),(1,-1),(2,0)$.
The function
\[g(u,v)=u^2+v^2-(t+2)u+v+(t+1)\]
is non-negative:
$g(1,-1)=0$,
$g(0,0)=1+t\ge0$,
$g(2,0)=1-t\ge0$.

Note that this discreteness proof does not work for $|t|>1$ as then we have either $g(0,0)<0$ or $g(2,0)<0$.
\end{itemize}

\item
Case $(n_1,n_2)=(4,4)$:
We list the values of~$a$, $b$ and~$a^2+b^2$ for all $p=\bar{\si}(0)$ with $\si\in\Si$:
\begin{center}
\def\arraystretch{1.5}
\begin{tabular}{|c|c|c|c|c| }
\hline$p$ & $a$ & $b$ & $a^2+b^2$ \\
\hline$0$ & $0$ & $0$ & $0$ \\
\hline\(j_1(0)\) & \(-\frac{1}{2}\left(t+1\right)\) & \(\frac{1}{2}\left(1-t\right)\) & \(\frac{1}{2}\left(t^2+1\right)\) \\
\hline\(j_{11}(0)\) & \(-1\) & \(-t\) & \(t^2+1\) \\
\hline\(j_{111}(0)\) & \(\frac{1}{2}\left(t-1\right)\) & \(-\frac{1}{2}\left(t+1\right)\) & \(\frac{1}{2}\left(t^2+1\right)\) \\
\hline
\end{tabular}
\end{center}
%We have $u\in(a-\ga,a+\ga)$, $v\in(b-\ga,b+\ga)$,
%where $\ga=\sqrt{3}-1$.
We calculate
\[g(u,v)=u^2+v^2-2au-2bv+(a^2+b^2)-\frac12(t^2+1).\]
We only need to consider
$u\in(a-\ga,a+\ga)$,
$v\in(b-\ga,b+\ga)$,
where $\ga=\sqrt{\frac{t^2+1}{2}}$.
%The assumption $|\al-\pi|\le\pi/6$ implies $|t|\le2-\sqrt{3}$.
We assume $|t|\le2-\sqrt{3}$.
\begin{itemize}
\item
$p=0$:
We consider $u,v\in(-1,1)$.
The only such $u,v\in\z$ with $u\equiv v\mod2$ are $u=v=0$,
which corresponds to the excluded case $f=\Id$.
\item
$p=j_1(0)$:
We consider $u\in(-2,1)$, $v\in(-1,2)$.
The only such $u,v\in\z$ with $u\equiv v\mod2$
are $u=-1$, $v=1$ and $u=v=0$.
The function
\[g(u,v)=u^2+v^2+\left(t+1\right)u+\left(t-1\right)v\]
is non-negative: $g(-1,1)=g(0,0)=0$.
\item
$p=j_{11}(0)$:
We consider $u\in(-2,0)$, $v\in(-2,2)$.
The only such $u,v\in\z$ with $u\equiv v\mod2$ are $(u,v)=(-1,-1),(-1,1)$.
The function
\[g(u,v)=u^2+v^2+2u+2tv+\frac12(t^2+1)\]
is non-negative:
For $|t|\le2-\sqrt{3}$,
\begin{align*}
  g(-1,\pm1)
%  &=\frac12(t^2\pm4t+1)\\
  &=\frac12\left(t\mp(2-\sqrt{3})\right)\left(t\pm(2+\sqrt{3})\right)
  \ge0.
\end{align*}

Note that this discreteness proof does not work for $|t|>2-\sqrt{3}$ as then we have either $g(-1,1)<0$ or $g(-1,-1)<0$.
\item
$p=j_{111}(0)$:
We consider $u,v\in(-2,1)$.
The only such $u,v\in\z$ with $u\equiv v\mod2$
are $(u,v)=(-1,1),(0,0)$.
The function
\[g(u,v)=u^2+v^2-\left(t-1\right)u+(t+1)v\]
is non-negative:
$g(-1,-1)=g(0,0)=0$.
\end{itemize}
\end{enumerate}
Hence in all cases all conditions of Lemma~\ref{f(0)} are satisfied and we can conclude that the group $\<\iota_1,\iota_2,\iota_3\>$ is discrete.
\end{proof}

\section{Non-Discreteness Results}
\label{nondisres}

\noindent
Let $\Ga=<\iota_1,\iota_2,\iota_3\>$
be an ultra-parallel $[m_1, m_2,0; n_1, n_2, 2]$-triangle group as in section~\ref{sec-param}.
In this section we will use the following theorems from~\cite{parkershimizu} (Proposition~8.2) and~\cite{parkerford} (Section~3.1) to find conditions for the group~$\Ga$ not to be discrete:

\begin{theorem}
\label{parkford}
Let $\Ga$ be a discrete subgroup of \(\PU(n,1)\) containing a vertical translation~$g$ by~$\nu>0$. Let $h$ be any element of~$\Gamma$ not fixing $q_\infty$, a distinguished point at infinity, and let 
\[r_h=\sqrt{\frac{2}{\abs{h_{22}-h_{23}+h_{32}-h_{33}}}}\]
be the radius of its isometric sphere. Then either 
\[\nu/r_h^2\ge2\quad\mbox{or}\quad \nu/r_h^2=2\cos(\pi/q) \quad\mbox{for some integer}\quad q\ge3.\]
\end{theorem}

\begin{theorem}
\label{parkshim}
Let \(\Gamma\) be a discrete subgroup of \(\PU(n,1)\) containing a Heisenberg translation by \((\xi, t)\) with \(\xi\not=0\) and a vertical translation by \(\nu>0\). Let \(h\) be any element of \(\Gamma\) not fixing \(q_\infty\) and let \(r_h\) be the radius of its isometric sphere. Then 
\[r_h^2 \leq \mbox{min}\left\{\nu, \abs{|\xi|^2+i\nu/2}\right\}.\]
\end{theorem}

\noindent
For the ultra-parallel \([m_1,m_2,0; n_1, n_2, 2]\)-triangle groups, we will apply Theorem \ref{parkford} to the vertical Heisenberg translation \(H = (0, \nu)\). We will then strengthen the results for the four ultra-parallel \([m_1,m_2,0; n_1, n_2, 2]\)-triangle groups, with $(n_1,n_2)=(3,3),(2,3),(2,6),(4,4)$, applying Theorem~\ref{parkshim} to a Heisenberg translations $(\xi, t)$ with $\xi\ne0$ and the vertical Heisenberg translation $H=(0,\nu)$. 
\\
\linebreak
Let $h=\iota_3$.
The matrix of the element \(h=\iota_3=\iota_3^{-1}\) is 
\[\begin{pmatrix} -1 & 0 & 0 \\ 0 & 1 & 0 \\ 0 & 0 & -1 \end{pmatrix}.\]
The radius of the isometric sphere of \(h\) is \(r_h = 1\). To see if the element \(h\) fixes \(q_\infty\), we first map \(\infty\) from the Heisenberg space to the boundary of the complex hyperbolic 2-space. That is,
\[\infty \mapsto [0:1:-1] \in \partial H_{\mathbb{C}}^2.\]
We apply \(h\) to this point
\[\begin{pmatrix} -1 & 0 & 0 \\ 0 & 1 & 0 \\ 0 & 0 & -1 \end{pmatrix}\begin{bmatrix} 0 \\ 1 \\ -1 \end{bmatrix} = \begin{bmatrix} 0 \\ 1 \\ 1 \end{bmatrix} =  [0:1:1] \in \partial H_{\mathbb{C}}^2.\]
Note that $h(\infty)\ne\infty$ and so $h=\iota_3$ does not fix~$q_\infty$.
Therefore, for Theorem \ref{parkford}, if the group is discrete then
\[\nu\ge2\quad\mbox{or}\quad \nu=2\cos(\pi/q) \quad\mbox{for some integer}\quad q\ge3.\]
Hence the group is not discrete if 
\begin{equation}
\label{nondis}
  \nu<2
  \quad\mbox{and}\quad
  \nu\ne2\cos(\pi/q)
  \quad\mbox{for some integer}
  \quad q\ge3.
\end{equation}

\subsection{Proof of Proposition~\ref{prop-nondiscr}}
As we are considering the ultra-parallel \([m_1, m_2, 0; n_1, n_2, 2]\)-triangle group, with \(m_1\ne m_2\), we have that \(R = \sqrt{r_1^2+r_2^2+2r_1r_2\cos(2\theta)}\). 
\begin{proof}
~~~
\begin{enumerate}[$\bullet$]
\item
$(n_1,n_2)=(2,4),(3,6)$:
For the vertical Heisenberg translation~$H$ we have $\nu=\nu_1R^2$,
where $\nu_1$ is determined by each triangle group:
\begin{align*}
&\nu_1=16\quad\text{for}~(n_1,n_2)=(2,4),\\
&\nu_1=2\sqrt{3}\quad\text{for}~(n_1,n_2)=(3,6).
\end{align*}
Substituting \(\nu=\nu_1R^2\) into (\ref{nondis}), we are able to conclude by Theorem \ref{parkford} that the group \(\Gamma\) is not discrete if
\begin{align*}
&\nu_1R^2<2
~\text{and}~
\nu_1R^2\ne2\cos(\pi/q)
~\text{for some integer}~q\ge3 \\
\Longleftrightarrow
&\cos^2(\theta)<\frac{2-\nu_1(r_1-r_2)^2}{4\nu_1r_1r_2}
~\mbox{and}~
\cos^2(\theta)\ne\frac{2\cos(\pi/q)-\nu_1(r_1-r_2)^2}{4\nu_1r_1r_2} 
\end{align*}
for some integer \(q\ge3\). Using 
\begin{equation}
\label{cos}
\cos^2(\theta) = \frac{1}{2}\left(\cos(2\theta)+1\right) = \frac{1}{2}\left(1-\cos(\alpha)\right)
\end{equation}
we are able to rewrite the inequalities to conclude that the group \(\Gamma\) is not discrete provided that
\[\cos(\alpha) > 1-\frac{2-\nu_1(r_1-r_2)^2}{2\nu_1r_1r_2} \quad\mbox{and}\quad \cos(\alpha) \not= 1- \frac{2\cos(\pi/q)-\nu_1(r_1-r_2)^2}{2\nu_1r_1r_2}\]
for some integer $q\ge3$.
Substituting the values of~$\nu_1$ will give the required limits for~$\cos(\alpha)$.
\item
$(n_1,n_2)=(3,3), (2,3),(2,6),(4,4)$:
For the vertical Heisenberg translation~$H$ we have \(\nu=\nu_1R^2\), where $\nu_1$ is determined by each triangle group:
\begin{align*}
&[m,m,0; 3,3,2], \quad \nu_1=6\sqrt{3} \\
&[m,m,0; 2,3,2], \quad \nu_1=24\sqrt{3} \\
&[m,m,0; 2,6,2], \quad \nu_1=4\sqrt{3} \\
&[m,m,0; 4,4,2], \quad \nu_1=4.
\end{align*}
Substituting \(\nu=\nu_1R^2\) into (\ref{nondis}), we are able to conclude by Theorem \ref{parkford} that the group \(\Gamma\) is not discrete provided that \[\cos(\alpha) > 1-\frac{2-\nu_1(r_1-r_2)^2}{2\nu_1r_1r_2} \quad\mbox{and}\quad \cos(\alpha) \not= 1- \frac{2\cos(\pi/q)-\nu_1(r_2-r_2)^2}{2\nu_1r_1r_2}\]
for some integer $q\ge3$.
\\
\linebreak
We will now use Theorem~\ref{parkshim} to further these non-discreteness results by eliminating some of the values of~$q$.
For the Heisenberg translation~$T_1$ by~$(\xi, t)$ with~$\xi\ne0$ and the vertical translation~$H$ by $\nu>0$, we can improve our results if we have 
\[\frac{\nu}{2}\leq \left||\xi|^2+\frac{i\nu}{2}\right|<\nu.\]
\\
\linebreak
For each of the triangle groups \([m_1,m_2,0; n_1,n_2,2]\) we have:
\begin{center}
\def\arraystretch{1.5}
\begin{tabular}{ c c c c }
\(\{n_1, n_2\}\) & \(\xi\) & \(\nu\) & \(\mbox{min}\left\{\nu,\abs{|\xi|^2+i\nu/2}\right\}\) \\
\hline
\(\{3, 3\}\) & \(i\cdot R\sqrt{3}\) & \(6\sqrt{3}R^2\) & \(6R^2\) \\ 
\(\{2, 3\}\) & \(i\cdot2R\sqrt{3}\) & \(24\sqrt{3}R^2\) & \(24R^2\) \\  
  \(\lbrace{2, 6\rbrace}\) & \(R\cdot(1+i\sqrt{3})\) & \(4\sqrt{3}R^2\) & \(2\sqrt{7}R^2\) \\ 
  \(\lbrace{4, 4\rbrace}\) & \(R\cdot(1+i)\) & \(4R^2\) & \(2\sqrt{2}R^2\) \\
\end{tabular}
\end{center}
Therefore, the triangle group is non-discrete when \[\nu_2R^2<1 \Leftrightarrow \cos^2(\theta)<\frac{1-\nu_2(r_1-r_2)^2}{4\nu_2r_1r_2}\]
for the cases
\begin{align*}
&[m,m,0; 3,3,2], \quad \nu_2=6 \\
&[m,m,0; 2,3,2], \quad \nu_2=24 \\
&[m,m,0; 2,6,2], \quad \nu_2=2\sqrt{7} \\
&[m,m,0; 4,4,2], \quad \nu_2=2\sqrt{2}
\end{align*}
Combining these results with those from Theorem \ref{parkford}, we can then eliminate the values \(q=3,4\) and \(5\) for the \([m_1,m_2,0; 3,3,2]\) and \([m_1,m_2,0; 2,3,2]\) cases. For the \([m_1,m_2,0; 2,6,2]\) and \([m_1,m_2,0; 4,4,2]\) cases we can eliminate the value \(q=3\).
\end{enumerate}
\end{proof}

\section{Acknowledgements} 
\noindent
We would like to thank John Parker for his insightful comments and suggestions throughout this work.

%\bibliographystyle{abbrv}
%\bibliography{../Template/biblio}

\begin{thebibliography}{KPTh}

\bibitem[1]{bb}
H.~Brown, R.~B\"ulow, J.~Neub\"user, H.~Wondratscheck and H.~Zassenhaus,
\emph{Crystallographic groups of four-dimensional Space},
Wiley, New York, 1978.

\bibitem[2]{dekimpe}
K.~Dekimpe,
\emph{Almost-Bieberbach Groups: Affine and Polynomial Structures},
Lecture Notes in Mathematics, vol.~1639, Springer-Verlag, 1996.

\bibitem[3]{goldman}
W.M.~Goldman,
\emph{{Complex Hyperbolic Geometry}},
Oxford University Press, 1999.

\bibitem[4]{hersonpaul}
S.~Hersonsky and F.~Paulin,
\emph{{On the Volumes of Complex Hyperbolic Manifolds}},
Duke Math. J, 84(3):719-737, 1996.

\bibitem[5]{andy}
A.~Monaghan,
\emph{{Complex Hyperbolic Triangle Groups}},
Ph.D. thesis, University of Liverpool, 2013.

\bibitem[6]{andyanna}
A.~Monaghan, J.R.~Parker and A,~Pratoussevitch,
\emph{Discreteness of Ultra-Parallel Complex Hyperbolic Triangle Groups of Type $[m_1,m_2,0]$},
Journal of the LMS, 100:545-567, 2019.

\bibitem[7]{parkershimizu}
J.R.~Parker,
\emph{Shimizu's Lemma for Complex Hyperbolic Space},
Int.\ Journal of Mathematics, 3:291--308, 1992.

\bibitem[8]{parkerford}
J.R.~Parker,
\emph{{On Ford isometric spheres in complex hyperbolic space}},
Math.\ Proc.\ Camb.\ Phil.\ Soc.\ 115:501--512, 1994.

\bibitem[9]{parkervol}
J.R.~Parker,
\emph{{On the Volumes of Cusped, Complex Hyperbolic Manifolds and Orbifolds}},
Duke Math. J, 94:433--464, 1998.

\bibitem[10]{parker2003}
J.R.~Parker,
\emph{{Notes on Complex Hyperbolic Geometry}},
lecture notes, 2010.

\bibitem[11]{parkpau}
J.R.~Parker and J.~Paupert
\emph{{Unfaithful Complex Hyperbolic Triangle Groups, II: Higher Order Reflections}},
Pacific Journal of Mathematics, 239:357--389, 2009.

\bibitem[12]{pov1}
S.~Povall,
\emph{{Ultra-Parallel Complex Hyperbolic Triangle Groups}},
Ph.D. thesis, University of Liverpool, 2019.

\bibitem[13]{povprat}
S.~Povall and A.~Pratoussevitch,
\emph{{Complex Hyperbolic Triangle Groups of Type $[m,m,0; 3,3,2]$}},
Conformal Geometry and Dynamics, 24:51--67, 2020.

\bibitem[14]{anna}
A.~Pratoussevitch,
\emph{{Traces in Complex Hyperbolic Triangle Groups}},
Geometriae Dedicata 111:159--185, 2005.

\bibitem[15]{schwartz}
R.E.~Schwartz,
\emph{Complex hyperbolic triangle groups},
Proceedings of the ICM (Beijing, 2002), Vol~II,
Higher Ed.\ Press, 339--349, 2002.

\bibitem[16]{WyssGall}
J.~Wyss-Gallifent,
\emph{{Complex Hyperbolic Triangle Groups}},
Ph.D. thesis, University of Maryland, 2000.

\end{thebibliography}

\appendix

\section{Classification of Almost-Crystallographic Groups}
\label{Dekimpe}

In this section we will outline an alternative approach to the understanding of the structure of the subgroup~$E=\<\iota_1,\iota_2\>$ using the classification of almost-crystallographic groups by Dekimpe~\cite{dekimpe}.

\begin{definition}
An \textit{almost-crystallographic group} is a uniform discrete subgroup~$\calE$ of $G\semiprod C$,
where $G$ is a connected, simply connected nilpotent Lie group and $C$ is a maximal compact subgroup of~$\Aut(G)$. 
\end{definition}

\noindent
As a discrete subgroup of $\mathpzc{N}\semiprod\U(1)$ (see section \ref{heisenberg-isometries}),
the group~$\calE$ is an almost-crystallographic group with $G=\calN$ and $\U(1)\subset C\subset\Aut(\calN)$.
The projection of~$\calE=\<\iota_1,\iota_2\>$ to~$\c$ is a wallpaper group $Q=\<j_1,j_2\>$,
where $j_k$ is the rotation of $\c$ around $\varphi_k$ through $2\pi/{n_k}$, for~$k=1,2$, obtained by projecting~$\iota_k$ to~$\c$.
\\
\linebreak
\textbf{The case $[m,m,0;2,3,2]$}.
The wallpaper group $Q=\<j_1, j_2\>$ is generated by two rotations of orders~$2$ and~$3$ respectively, and has a presentation
\[
  \<j_1,j_2\:|\: j_1^2=j_2^3=\left(j_{12}\right)^6=1\>.
\]
The standard notation for this wallpaper group is \textbf{p6}; see for example~\cite{bb}.
In the classification of three-dimensional almost-crystallographic groups in section \(7.1\) of \cite{dekimpe}, the wallpaper group \textbf{p6} appears in case \(16\) on page \(166\). In this case the group~$\calE$ is generated by elements $a,b,c,\alpha$ with relations
\[
  [b,a]=c^{k_1}, 
  [c,a]=[c,b]=[c,\alpha]=1, \:
  \alpha a=ab\alpha c^{k_2}, \:
  \alpha b=a^{-1}\alpha c^{k_3}, \:
  \alpha^6=c^{k_4}.
\]
We consider the generators $\iota_1=\alpha^3 a$ and $\iota_2=\alpha^2$,
so that $\left(\alpha^3 a\right)^2=\alpha^6=1$.
The hypothesis $\alpha^6=1$ implies that $k_4=0$.
The hypothesis $\left(\alpha^3 a\right)^2=1$ can be rewritten as 
\[
  c^{-k_1+2k_2+2k_3}=1
  \Rightarrow k_1=2\left(k_2+k_3\right).
\]
The translations~$T_1$ and~$T_2$ in Proposition~\ref{structure-E} are 
\begin{align*}
T_1=\iota_{21212} &= \alpha^2\left(\alpha^3a\right)\alpha^2\left(\alpha^3a\right)\alpha^2 = b^{-2}a^{-1}c^{-k_2+2k_3}; \\
T_2=\iota_{12212} &= \alpha^3a\alpha^2\alpha^2\left(\alpha^3a\right)\alpha^2 = a^{-1}b^{-1}a^{-1}c^{-k_2}.
\end{align*}
Their commutator is
\begin{align*}
  &H=[T_1, T_2]\\
  &=\left(b^{-2}a^{-1}c^{-k_2+2k_3}\right)^{-1}\cdot \left(a^{-1}b^{-1}a^{-1}c^{-k_2}\right)^{-1}\cdot b^{-2}a^{-1}c^{-k_2+2k_3}\cdot a^{-1}b^{-1}a^{-1}c^{-k_2} \\
  &= ab^2abab^{-1}a^{-2}b^{-1}a^{-1} = c^{3k_1}.
\end{align*}
On the other hand, the kernel of the map $E=\<\iota_1,\iota_2\>\rightarrow\<j_1,j_2\>$ given by $\iota_1\mapsto j_1$, $\iota_2\mapsto j_2$ is generated by $\left(\iota_{12}\right)^6$. Calculating $\left(\iota_{12}\right)^6$
we obtain $c^{6\left(k_2+k_3\right)}$.
Using $k_1=2\left(k_2+k_3\right)$ we can rewrite this as $\left(\iota_{12}\right)^6=c^{3k_1}$. Hence the element
$
 H=[T_1,T_2]
 =\left(\iota_{12}\right)^6
 =c^{3k_1}
$
is the shortest vertical Heisenberg translation in $E=\<\iota_1,\iota_2\>$.
\\
\linebreak
\textbf{The case \([m,m,0;2,4,2]\)}.
The wallpaper group $Q=\<j_1,j_2\>$ is generated by two rotations of orders~$2$ and~$4$ respectively,
and has a presentation
\[
  \<j_1,j_2 \:|\: j_1^2=j_2^4=\left(j_{12}\right)^4=1\>.
\]
The standard notation for this wallpaper group is~\textbf{p4}.
In the classification of three-dimensional almost-crystallographic groups in section~7.1 of~\cite{dekimpe}, the wallpaper group~\textbf{p4} appears in case~10 on page~163.
In this case the group~$\calE$ is generated by elements $a,b,c,\alpha$ with relations
\[[b,a]=c^{k_1}, [c,a]=[c,b]=[c,\alpha]=1, \:\alpha a=b\alpha c^{k_2}, \:\alpha b=a^{-1}\alpha c^{k_3}, \:\alpha^4=c^{k_4}.\]
We consider the generators $\iota_1=\alpha^2 a$ and $\iota_2=\alpha$, so that $\left(\alpha^2 a\right)^2=\alpha^4=1$.
The hypothesis $\alpha^4=1$ implies that $k_4=0$.
The hypothesis $\left(\alpha^2 a\right)^2=1$ can be rewritten as 
\[c^{k_2+k_3}=1\Rightarrow k_2+k_3=0.\]
The translations~$T_1$ and~$T_2$
in Proposition~\ref{structure-E} are 
\begin{align*}
T_1=\iota_{212}
&=\alpha\left(\alpha^2a\right)\alpha=b^{-1}c^{k_3};\\
T_2=\iota_{122}
&=\alpha^2a\alpha^2 = a^{-1}.
\end{align*}
Their commutator is
\begin{align*}
H
=[T_1, T_2]
&=\left(b^{-1}c^{k_3}\right)^{-1}\cdot \left(a^{-1}\right)^{-1}\cdot b^{-1}c^{k_3}\cdot a^{-1} \\
&=bab^{-1}a^{-1}=abc^{k_1}b^{-1}a^{-1} = c^{k_1}.
\end{align*}
On the other hand, the kernel of the map $\calE=\<\iota_1,\iota_2\>\rightarrow\<j_1, j_2\>$
given by $\iota_1\mapsto j_1,\iota_2\mapsto j_2$ is generated by $\left(\iota_{12}\right)^4$.
Calculating $\left(\iota_{12}\right)^4$ we obtain $c^{k_1}$.
Hence the element
$H=[T_1,T_2]=\left(\iota_{12}\right)^4=c^{k_1}$
is the shortest vertical Heisenberg translation in $\calE=\<\iota_1,\iota_2\>$.
\\
\linebreak
\textbf{The case \([m,m,0; 4,4,2]\)}.
The wallpaper group $Q=\<j_1, j_2\>$ is generated by two rotations of order~$4$, and has a presentation
\[\<j_1, j_2 \:|\: j_1^4=j_2^4=\left(j_{12}\right)^2=1\>.\]
The standard notation for this wallpaper group is~\textbf{p4} (for more details, see \\
\([m,m,0; 2,4,2]\)-case above). We consider the generators $\iota_1=\alpha a$ and $\iota_2=\alpha$,
so that $\left(\alpha a\right)^4=\alpha^4=1$.
The hypothesis $\alpha^4=1$ implies that $k_4=0$.
The hypothesis $\left(\alpha a\right)^4=1$ can be rewritten as 
\[
  c^{-k_1+2k_2+2k_3}=1
  \Rightarrow
  k_1=2\left(k_2+k_3\right).
\]
The translations~$T_1$ and~$T_2$ in Proposition~\ref{structure-E} are 
\begin{align*}
T_1=\iota_{1112}
&=\alpha a\alpha a\alpha a\alpha
=ba^{-1}b^{-1}c^{k_1}; \\
T_2=\iota_{2111}
&=\alpha^2a\alpha a\alpha a
=a^{-1}b^{-1}ac^{k_1-k_2}.
\end{align*}
Their commutator is
\begin{align*}
H^2&=[T_2, T_1]\\
&=\left(a^{-1}b^{-1}ac^{k_1-k_2}\right)^{-1}\cdot \left(ba^{-1}b^{-1}c^{k_1}\right)^{-1}\cdot a^{-1}b^{-1}ac^{k_1-k_2}\cdot ba^{-1}b^{-1}c^{k_1} \\
&=a^{-1}babab^{-1}a^{-1}b^{-1}aba^{-1}b^{-1} = c^{k_1}.
\end{align*}
On the other hand, the kernel of the map \(E=\langle{\iota_1, \iota_2\rangle} \rightarrow \langle{j_1, j_2\rangle}\) given by \(\iota_1 \rightarrow j_1, \iota_2 \rightarrow j_2\) is generated by \(\left(\iota_{12}\right)^2\). Calculating \(\left(\iota_{12}\right)^2\) we obtain \(c^{k_2+k_3}\). Using \(k_1=2\left(k_2+k_3\right)\) we can rewrite this as \(\left(\iota_{12}\right)^2 = c^{\frac{k_1}{2}}\). Hence the element \(H = \left(\iota_{12}\right)^2 = c^{\frac{k_1}{2}}\) is the shortest vertical Heisenberg translation in \(E=\langle{\iota_1, \iota_2\rangle}\).
\\
\linebreak
\textbf{The case \([m,m,0; 2,6,2]\)}. The wallpaper group \(Q=\langle{j_1, j_2\rangle}\) is generated by two rotations of orders \(2\) and \(6\) respectively, and has a presentation
\[\langle{j_1, j_2 \:|\: j_1^2=j_2^6=\left(j_{12}\right)^3=1\rangle}.\]
The standard notation for this wallpaper group is \textbf{p6}; (for more details, see \\
\([m,m,0; 2,3,2]\)-case above). We consider the generators \(\iota_1=\alpha^3 a\) and \(\iota_2=\alpha\), so that \(\left(\alpha^3 a\right)^2=\alpha^6=1.\) The hypothesis \(\alpha^6=1\) implies that \(k_4=0\). The hypothesis \(\left(\alpha^3 a\right)^2=1\) can be rewritten as 
\[c^{-k_1+2k_2+2k_3}=1 \Rightarrow k_1=2\left(k_2+k_3\right).\]
The translations~$T_1$ and~$T_2$ in Proposition~\ref{structure-E} are 
\begin{align*}
T_1=\iota_{2122} &= \alpha\left(\alpha^3a\right)\alpha^2 = b^{-1}a^{-1}c^{-k_2}; \\
T_2=\iota_{2212} &= \alpha^2\left(\alpha^3a\right)\alpha = b^{-1}c^{k_3}.
\end{align*}
Their commutator is \(H^2=[T_2, T_1]\)
\begin{align*}
&= \left(b^{-1}c^{k_3}\right)^{-1}\cdot \left(b^{-1}a^{-1}c^{-k_2}\right)^{-1}\cdot b^{-1}c^{k_3} \cdot b^{-1}a^{-1}c^{-k_2} \\
&= bab^{-1}a^{-1} = c^{k_1}.
\end{align*}
On the other hand, the kernel of the map \(E=\langle{\iota_1, \iota_2\rangle} \rightarrow \langle{j_1, j_2\rangle}\) given by \(\iota_1 \rightarrow j_1, \iota_2 \rightarrow j_2\) is generated by \(\left(\iota_{12}\right)^3\). Calculating \(\left(\iota_{12}\right)^3\) we obtain \(c^{k_2+k_3}\). Using \(k_1=2\left(k_2+k_3\right)\) we can rewrite this as \(\left(\iota_{12}\right)^3 = c^{\frac{k_1}{2}}\). Hence the element \(H= \left(\iota_{12}\right)^3 = c^{\frac{k_1}{2}}\) is the shortest vertical Heisenberg translation in \(E=\langle{\iota_1, \iota_2\rangle}\).
\\
\linebreak
\textbf{The case \([m,m,0; 3,6,2]\)}. The wallpaper group \(Q=\langle{j_1, j_2\rangle}\) is generated by two rotations of orders \(3\) and \(6\) respectively, and has a presentation
\[\langle{j_1, j_2 \:|\: j_1^3=j_2^6=\left(j_{12}\right)^2=1\rangle}.\]
The standard notation for this wallpaper group is \textbf{p6}; (for more details, see 
\\
\([m,m,0; 2,3,2]\)-case above). We consider the generators \(\iota_1=\alpha^2 a\) and \(\iota_2=\alpha\), so that \(\left(\alpha^2 a\right)^3=\alpha^6=1.\) The hypothesis \(\alpha^6=1\) implies that \(k_4=0\). The hypothesis \(\left(\alpha^2 a\right)^3=1\) can be rewritten as 
\[c^{-2k_1+3k_2+3k_3}=1 \Rightarrow 2k_1=3\left(k_2+k_3\right).\]
The translations \(T_1\) and \(T_2\) in Proposition \ref{structure-E} are 
\begin{align*}
T_1=\iota_{1122} &= \alpha^2 a\alpha^2a\alpha^2 = a^{-1}; \\
T_2=\iota_{2211} &= \alpha^2\left(\alpha^2a\right)\alpha^2a = b^{-1}c^{-k_1+k_2+2k_3}.
\end{align*}
Their commutator is \(H^3=[T_2, T_1]\)
\begin{align*}
&= \left(b^{-1}c^{-k_1+k_2+2k_3}\right)^{-1}\cdot \left(a^{-1}\right)^{-1}\cdot b^{-1}c^{-k_1+k_2+2k_3} \cdot a^{-1} \\
&= bab^{-1}a^{-1} = c^{k_1}.
\end{align*}
On the other hand, the kernel of the map \(E=\langle{\iota_1, \iota_2\rangle} \rightarrow \langle{j_1, j_2\rangle}\) given by \(\iota_1 \rightarrow j_1, \iota_2 \rightarrow j_2\) is generated by \(\left(\iota_{12}\right)^2\). Calculating \(\left(\iota_{12}\right)^2\) we obtain \(c^{\frac{1}{2}\left(k_2+k_3\right)}\). Using \(2k_1=3\left(k_2+k_3\right)\) we can rewrite this as \(\left(\iota_{12}\right)^2 = c^{k_1/3}\).
Hence the element
$H=\left(\iota_{12}\right)^2=c^{k_1/3}$
is the shortest vertical Heisenberg translation in $E=\<\iota_1, \iota_2\>$.

\end{document}